\documentclass[11pt]{amsart}

\calclayout

\usepackage{color}
\usepackage{verbatim}
\usepackage[noadjust]{cite}

\usepackage[utf8x]{inputenc}
\usepackage{t1enc}
\usepackage[english]{babel}
\usepackage{pgfplots}
\pgfplotsset{compat=1.15}
\usepackage{mathrsfs}
\usetikzlibrary{arrows}

\definecolor{ffqqqq}{rgb}{1.,0.,0.}
\definecolor{uuuuuu}{rgb}{0.26666666666666666,0.26666666666666666,0.26666666666666666}

\usepackage{amsmath,amssymb,amsthm}
\usepackage{mathrsfs,esint}
\usepackage{graphicx,wrapfig}
\usepackage[format=hang]{caption}
\newcommand{\R}{\mathbb{R}}

\newcommand{\N}{\mathbb{N}}

\newcommand{\F}{\mathcal{F}}
\newcommand{\A}{\mathcal{A}}

\newcommand{\M}{\mathcal{M}}
\newcommand{\No}{\mathcal{N}}

\newcommand{\J}{\mathcal{J}}

\newcommand{\Ha}{\mathcal{H}}

\newcommand{\eps}{\varepsilon}

\newcommand{\loc}{\text{loc}}

\newcommand{\phii}{\varphi}

\newcommand{\bmat}{\begin{bmatrix}}
\newcommand{\emat}{\end{bmatrix}}

\newcommand{\wtil}{\widetilde}

\newcommand{\JOm}{\mathcal{J}_\Omega}
\newcommand{\st}{\text{ s.t. }}

\providecommand*{\vint}[1]{\mathchoice
          {\mathop{\vrule width 5pt height 3 pt depth -2.5pt
                  \kern -9pt \kern 1pt\intop}\nolimits_{\kern -5pt{#1}}}
          {\mathop{\vrule width 5pt height 3 pt depth -2.6pt
                  \kern -6pt \intop}\nolimits_{\kern -3pt{#1}}}
          {\mathop{\vrule width 5pt height 3 pt depth -2.6pt
                  \kern -6pt \intop}\nolimits_{\kern -3pt{#1}}}
          {\mathop{\vrule width 5pt height 3 pt depth -2.6pt
                  \kern -6pt \intop}\nolimits_{\kern -3pt{#1}}}}

\DeclareMathOperator{\Mod}{Mod}

\DeclareMathOperator{\dist}{dist}
\DeclareMathOperator{\diam}{diam}
\DeclareMathOperator{\rad}{rad}
\DeclareMathOperator{\supt}{supp}

\DeclareMathOperator{\codim}{codim}

\numberwithin{equation}{section}
\theoremstyle{plain}
\newtheorem{thm}[equation]{Theorem}

\newtheorem{prop}[equation]{Proposition}
\newtheorem{cor}[equation]{Corollary}
\newtheorem{lem}[equation]{Lemma}

\theoremstyle{definition}

\newtheorem{defn}[equation]{Definition}

\newtheorem{remark}[equation]{Remark}

\newtheorem{example}[equation]{Example}

\makeatletter
\def\blfootnote{\xdef\@thefnmark{}\@footnotetext}
\makeatother

\begin{document}

\title[Regularity of sets of finite fractional perimeter]{Regularity of sets of finite fractional perimeter and nonlocal minimal surfaces in metric measure spaces}

\blfootnote{2020 {\it Mathematics Subject Classification.} 46E36,\,26A45,\,49Q20}
\blfootnote{{\it Keywords and phrases.}  fractional perimeter, Besov spaces, nonlocal minimal surfaces, metric measure space}

\author{Josh Kline}
\address{Department of Mathematical Sciences, P.O. Box 210025, University of Cincinnati,
Cincinnati, OH 45221–0025, U.S.A.}
\email{klinejp@ucmail.uc.edu}

\date{\today}

\maketitle

\begin{abstract}
In the setting of a doubling metric measure space, we study regularity of sets with finite $s$-perimeter, that is, sets whose characteristic functions have finite Besov energy, with regularity parameter $0<s<1$ and exponent $p=1$.  Following a result of Visintin in $\R^n$, we provide a sufficient condition for finiteness of the $s$-perimeter given in terms of the upper Minkowski codimension of the regularized boundary of the set.  We also show that if a set has finite $s$-perimeter, then its measure-theoretic boundary has codimension $s$ Hausdorff measure zero.  To the best of our knowledge, this result is new even in the Euclidean setting.  By studying certain fat Cantor sets, we provide examples illustrating that the converses of these results do not hold in general. 
In the doubling metric measure space setting, we then consider minimizers of a nonlocal perimeter functional, extending the definition introduced by Caffarelli, Roquejoffre, and Savin in $\R^n$, and prove existence, uniform density, and porosity results for minimizers.   
\end{abstract}

\section{Introduction}

In the setting of a doubling metric measure space $(X,d,\mu)$, we consider measurable sets $E\subset X$ for which the following Besov energy is finite for some fractional parameter $0<s<1$:
\[
\|\chi_E\|_{B^s_{1,1}(X)}=\int_X\int_X\frac{|\chi_E(x)-\chi_E(y)|}{d(x,y)^s\mu(B(x,d(x,y)))}d\mu(y)d\mu(x).
% =2\int_E\int_{X\setminus E}\frac{d\mu(y)d\mu(x)}{d(x,y)^s\mu(B(x,d(x,y)))}.
\]
Specifically, we study regularity properties of these sets of finite $s$-perimeter and minimizers of the $s$-perimeter functional with respect to a given bounded domain $\Omega$.  

In the celebrated work of Bourgain, Brezis, and Mironescu \cite{BBM}, it was shown that 
 for $p>1$, functions in $W^{1,p}(\R^n)$ are characterized by the limiting behavior of the Besov $B^s_{p,p}$-energy, under suitable rescaling, as $s\to 1^-$.  It was then shown by D\'avila \cite{D} that the same characterization holds for functions in $BV(\R^n)$ when $p=1$.  That is, given a measurable set $E\subset\R^n$, 
 \[ \lim_{s\to1^-}(1-s)\|\chi_E\|_{B^s_{1,1}(\R^n)}=C_nP(E,\R^n),
 \]
 where $P(E,\R^n)$ is the perimeter of $E$, given by the BV energy of $\chi_E$, and $C_n$ is a dimensional constant.  More recently, Di Marino and Squassina \cite{DS} proved analogous results in the setting of a complete, doubling metric measure space supporting a $(1,p)$-Poincar\'e inequality.  When $p=1$, they showed that 
 \[
 P(E,X)\lesssim\liminf_{s\to 1^-}(1-s)\|\chi_E\|_{B^s_{1,1}(X)}\le\limsup_{s\to 1^-}(1-s)\|\chi_E\|_{B^s_{1,1}(X)}\lesssim P(E,X).
 \]
Hence, in both Euclidean and metric settings, sets of finite perimeter are characterized by the limiting behavior of their $s$-perimeters.  For more on characterization of Sobolev and BV functions in the metric setting by nonlocal functionals, see \cite{G,LPZ1,LPZ2,L1}.

Federer's characterization \cite{F} states that $E\subset\R^n$ is a set of finite perimeter if and only if $\Ha^{n-1}(\partial^* E)<\infty$, where $\partial^*E$ is the measure-theoretic boundary of $E$, see Section~\ref{sec:Boundaries and Measures}.  In the setting of a complete, doubling metric measure space supporting a $(1,1)$-Poincar\'e inequality, the ``only if'' direction was proved by Ambrosio \cite{A}, and more recently the ``if'' direction was shown by Lahti \cite{L2}.  Here the codimension 1 Hausdorff measure $\Ha^{-1}$ was used as a replacement for $\Ha^{n-1}$, see Section~\ref{sec:Boundaries and Measures}.  In fact, it was shown in \cite{A,L3} that this characterization holds in this setting with $\partial^*E$ replaced with a smaller subset of the boundary, namely points where $E$ and its complement have lower density bounded below by a constant which depends only on the doubling and Poincar\'e constants, see Section~\ref{sec:Poincare inequalities}.  As sets of finite perimeter are characterized both by this boundary regularity and the limiting behavior of their $s$-perimeters, we first consider in this paper the relationship between the $s$-perimeter of a set and codimension $s$ measurements of its boundary. 

The von Koch snowflake domain $K\subset\R^2$ shows that the ``if'' direction of Federer's characterization does not hold in the fractional case.  Indeed, it was shown in \cite{Lom} that $\chi_K\in B^s_{1,1}(\R^2)$ if and only if $s\in (0,2-\log4/\log3)$.  Since $\partial K=\partial^*K$ and $\Ha^{\log4/\log3}(\partial K)<\infty$ (see \cite{Fa} for example), we see that $\Ha^{n-s}(\partial^*E)<\infty$ does not imply that $E$ is a set of finite $s$-perimeter in general.  However, considering the upper Minkowski content rather than the Hausdorff measure provides a sufficient condition for finiteness of the $s$-perimeter.  In \cite{V}, the following fractional codimension was introduced by Visintin for bounded sets $E\subset\R^n$,
\[
\codim_{F}(E)=\sup\{0<s<1:\chi_E\in B^s_{1,1}(\R^n)\},
\]
and it was shown that this codimension is bounded below by the upper Minkowski codimension of the regularized boundary $\partial^+E$ of $E$.  The regularized boundary contains the measure-theoretic boundary, and any set can be modified on a set of measure zero so that its topological boundary coincides with its regularized boundary, see Lemma~\ref{lem:Set Representative}.  See Section~\ref{sec:Boundaries and Measures} for the precise definitions of the regularized boundary and Minkowski contents and codimensions.  We first generalize this result from \cite{V} to the doubling metric measure space setting.  For the following and subsequent result, we denote our ambient metric measure space by $(Z,d,\nu)$, rather than $(X,d,\mu)$, in order to avoid confusion with the hyperbolic filling, used in the same section.

\begin{thm}\label{thm:MinkowskiSufficient}
   Let $(Z,d,\nu)$ be a metric measure space satsifying the LLC-1 condition, with $\nu$ a doubling measure. Let $E\subset Z$ be a bounded measurable set, and suppose that there exists $t>0$ such that $\overline{\mathcal{M}}^{-t}(\partial^+E)<\infty$.  Then $\chi_E\in B^s_{1,1}(Z)$ for all $0<s<t$.
\end{thm}
\noindent Here, $\overline\M^{-t}$ is the codimension $t$ upper Minkowski content, defined in Section~\ref{sec:Boundaries and Measures}, and the LLC-1 condition, defined in Section~\ref{sec:Doubling}, is a connectedness condition on $Z$.  In Example~\ref{ex:Mink LLC}, we show that the conclusion to Theorem~\ref{thm:MinkowskiSufficient} may fail if $Z$ does not satisfy this condition.  Furthermore, by constructing certain fat Cantor sets, we show in Example~\ref{ex:Mink Converse} that the converse to Theorem~\ref{thm:MinkowskiSufficient} does not hold in general.  In fact, we provide an example of a set $E\in B^s_{1,1}(Z)$ for which $\partial^*E\subsetneq\partial^+E$ and $\underline\M^{-t}(\partial^*E)=\infty$ for all $0<t<1$.  Thus, the converse does not hold even with the weaker conclusion of $\underline\M^{-t}(\partial^*E)<\infty$. Here $\underline\M^{-t}$ is the codimension $t$ lower Minkowski content, see Section~\ref{sec:Boundaries and Measures}.  For related results in the setting of strongly local Dirichlet metric spaces with sub-Gaussian heat kernel estimates, see \cite[Section~5.3]{ABCKRST}.

We then study necessary conditions for finiteness of the $s$-perimeter in the metric setting, and obtain the following result:

\begin{thm}\label{thm:CodimHausZero}
Let $(Z,d,\nu)$ be a compact metric measure space with $\nu$ a doubling measure, and let $0<s<1$.  If $E\subset Z$ is such that $\chi_E\in B^s_{1,1}(Z)$, then $\Ha^{-s}(\partial^*E)=0$. 
\end{thm}
\noindent Here, $\Ha^{-s}$ is the codimension $s$ Hausdorff measure, see Section~\ref{sec:Boundaries and Measures}.  To the best of our knowledge, this result is new even in the Euclidean setting. To prove this result, we utilize recent developments from analysis in metric spaces pertaining to the hyperbolic filling of a compact, doubling metric measure space.  In \cite{BBS}, it was shown that any compact, doubling metric measure space $(Z,d,\nu)$ can be realized as the boundary of a uniform metric measure space $X$ which is doubling and supports a $(1,1)$-Poincar\'e inequality.  
 Furthermore, bounded trace and extension operators exist between the Besov spaces $B^s_{p,p}(Z)$ and the Newton-Sobolev spaces $N^{1,p}(X)$.  See Theorem~\ref{thm:HypFillThm} and Section~\ref{sec:Hyperbolic Filling} below for an outline of the construction of the hyperbolic filling.  To study the measure-theoretic properties of $E\subset Z$, we consider the Newton-Sobolev extension of $\chi_E$ to the hyperbolic filling of $Z$.  There, the validity of a $(1,1)$-Poincar\'e inequality provides us with useful potential theoretic tools with which to study $\partial^*E$.  As with Theorem~\ref{thm:MinkowskiSufficient}, we also construct certain fat Cantor sets in Example~\ref{ex:Haus Converse} which illustrate that the converse of Theorem~\ref{thm:CodimHausZero} does not hold in general.  This construction also shows that Theorem~\ref{thm:CodimHausZero} is sharp in the sense that we cannot replace $\Ha^{-s}(\partial^*E)=0$ with $\Ha^{-t}(\partial^+E)<\infty$.  Indeed,  we show in Example~\ref{ex:Haus Sharp} that there exists a fat Cantor set $C_a$ such that $\chi_{C_a}\in B^s_{1,1}([0,1])$, but $\Ha^{-t}(\partial^+C_a)=\infty$ for all $0<t<1$. 

 By combining Theorem~\ref{thm:MinkowskiSufficient} and Theorem~\ref{thm:CodimHausZero}, we obtain the following corollary, which generalizes the result from \cite{V} to the metric setting and provides a new upper bound in terms of the Hausdorff codimension.  See Section~\ref{sec:Boundaries and Measures} and \ref{sec:Fractional Perimeter} for the precise definitions.

 \begin{cor}\label{cor:Visintin Improvement}
    Let $(Z,d,\nu)$ be a compact metric measure space satisfying the LLC-1 condition, with $\nu$ a doubling measure.  Let $E\subset Z$ be measurable.  Then,
    \begin{equation}\label{eq:Codim Ineq}
    \overline\codim_\M(\partial^+ E)\le\codim_F(E)\le\codim_\Ha(\partial^*E).
    \end{equation}
\end{cor}
\noindent In \cite{Lom}, it was shown that \eqref{eq:Codim Ineq} holds with equality when $E\subset\R^2$ is the von Koch snowflake domain.  In Example~\ref{ex:Codim Strict}, we show that these inequalities may be strict, again by constructing a certain fat Cantor set. 

In the second half of this paper, we consider a complete, doubling metric measure space $(X,d,\mu)$, and fix a bounded domain $\Omega\subset X$ and $0<s<1$.  We then study minimizers of the $\J_\Omega^s$ functional, defined here in Section~\ref{sec:Fractional Perimeter}, as originally introduced in $\R^n$ by Caffarelli, Roquejoffre, and Savin \cite{CRS}.  Related to problems arising in phase transitions, motion by mean curvature, and front propogation, see for example \cite{ABS,I,Ish,IPS,Sl}, this functional measures the interaction between a set and its complement given by the $B^s_{1,1}$-seminorm, but excludes the interaction occurring in the complement of $\Omega$.  In \cite{CRS}, existence and regularity of minimizers, referred to there as \emph{nonlocal minimal surfaces}, were studied for Lipschitz domains $\Omega\subset\R^n$.  In particular, it was shown that if $E\subset\R^n$ is a minimizer, then $\partial E\cap\Omega$ is a $C^{1,\alpha}$-hypersurface around each of its points, outside of an $n-2$ dimensional set.  Since their introduction in \cite{CRS}, nonlocal minimal surfaces have been studied at great length and remain a subject of active research in the Euclidean setting, see for example \cite{CV,DSV,DSV2,FV,MRT} and references therein.  Of particular interest are asymptotic results as the nonlocal parameter $s$ converges to either $0$ or $1$ and the ``stickiness'' phenomenon which is exhibited by nonlocal minimal surfaces, see for example \cite{DV,BLV,B,DSV3,DOE}.  For $\Gamma$-convergence results regarding the $\J_\Omega^s$ functional, see \cite{ADM}.

We consider this problem in the doubling metric measure space setting, and first prove an existence result via the direct method of calculus of variations in Theorem~\ref{thm:Existence of s-minimizers}.  To do so, we first prove a compactness result in Theorem~\ref{thm:BesovCompactness} by adapting the proof of \cite[Theorem~7.1]{Hitch} to the metric setting.  We then obtain the following uniform density result, which was established in \cite[Theorem~4.1]{CRS} in the Euclidean setting:

\begin{thm}\label{thm:Uniform Density}
    Let $(X,d,\mu)$ be a connected doubling metric measure space, let $\Omega\subset X$ a bounded domain and let $0<s<1$. Then there exists a constant $\gamma_0>0$, such that if $E\subset X$ is a minimizer of $\J_\Omega^s$, then after possibly modifying $E$ on a set of measure zero, it follows that for all $x_0\in\partial E\cap\Omega$ and $R_0>0$ such that $B(x_0,2R_0)\subset\Omega$, we have 
    \begin{equation}\label{eq:Uniform Density}
        \frac{\mu(B(x_0,R_0)\cap E)}{\mu(B(x_0,R_0))}\ge\gamma_0\quad\text{and}\quad\frac{\mu(B(x_0,R_0)\setminus E)}{\mu(B(x_0,R_0))}\ge\gamma_0.
    \end{equation}
    Here, the constant $\gamma_0$ depends only on $s$ and the doubling constant of $\mu$, see \eqref{eq:Density}.
\end{thm}
\noindent The connectedness assumption on $X$ can be relaxed, see Remark~\ref{rem:Connectedness}.  Our proof is inspired by the arguments of \cite[Theorem~4.2]{KKLS}, where an analogous uniform density result was obtained for minimizers of the (local) perimeter functional, in the setting of a doubling metric measure space supporting a $(1,1)$-Poincar\'e inequality.  There, the result is obtained via the De Giorgi method, where the Poincar\'e inequality plays a crucial role.  Since we are dealing with the fractional case, we do not require the validity of a Poincar\'e inequality, instead utilizing a fractional Poincar\'e inequality \cite[Theorem~3.4]{DLV}, which is valid without additional assumptions on $X$. 

As a consequence of Theorem~\ref{thm:Uniform Density}, we obtain the following porosity results for minimizers.  In the Euclidean setting, this was established in \cite[Corollary~4.3]{CRS}.

\begin{thm}\label{thm:Porosity}  Let $(X,d,\mu)$ be a doubling metric measure space satisfying the LLC-1 condition, let $\Omega\subset X$ be a bounded domain, and let $0<s<1$.  Then there exists a constant $C\ge 1$ such that if $E$ is a minimizer of $\mathcal{J}^s_\Omega$, then after possibly modifying $E$ on a set of measure zero, it follows that for all $x_0\in\partial E\cap\Omega$ and $R_0>0$ such that $B(x_0,2R_0)\subset\Omega,$ there exist $y,z\in B(x_0,R_0)$ such that 
\begin{equation}\label{eq:Porosity}
    B(y,R_0/C)\subset E\cap\Omega\text{ and }B(z,R_0/C)\subset\Omega\setminus E.
\end{equation}
The constant $C$ depends only on $s$, the doubling constant of $\mu$, and the constant $C_L$ from the LLC-1 condition.
\end{thm}
\noindent We point out that Theorem~\ref{thm:Uniform Density} follows also from Theorem~\ref{thm:Porosity}. 
 However, to prove Theorem~\ref{thm:Porosity}, we use the uniform density result of Theorem~\ref{thm:Uniform Density}, and so we prove it independently.  As in Theorem~\ref{thm:Uniform Density}, our proof is inspired by the arguments used to obtain the corresponding porosity results for minimizers of the local perimeter functional in \cite[Theorem~5.2]{KKLS}. As before, however, the argument there utilizes the $(1,1)$-Poincar\'e inequality, which we do not assume in our case.  As a corollary to Theorem~\ref{thm:Porosity}, we see that if $E$ is a minimizer of $\J_\Omega^s$ with $\chi_E\in B^s_{1,1}(X)$, then $\partial E\cap\Omega\subset\partial^*E$, and so by Theorem~\ref{thm:CodimHausZero}, we have that $\Ha^{-s}(\partial E\cap\Omega)=0$, see Corollary~\ref{cor:Minimizer Haus}. 

The structure of the paper is as follows: in Section~\ref{sec:Prelim} we introduce the necessary definitions and background notions, and in Section~\ref{sec:Size of Boundary}, we prove Theorems~\ref{thm:MinkowskiSufficient} and \ref{thm:CodimHausZero}.  In Section~4, we establish examples illustrating sharpness of the results of Section~\ref{sec:Size of Boundary} and that the converses of these results do not hold. In Section~\ref{sec:Existence of Minimizers}, we prove existence results for minimizers of $\J_\Omega^s$, and in Section~\ref{sec:Regularity of Minimizers}, we prove Theorems~\ref{thm:Uniform Density} and \ref{thm:Porosity}.

\section*{Acknowledgements}
 This research was partially supported by the NSF Grant DMS \#2054960, as well as the University of Cincinnati's University Research Center summer grant and the Taft Research Center Dissertation Fellowship Award.  The author would like to thank Nageswari Shanmugalingam for many valuable discussions regarding this topic and for her kind encouragement and support.  We would also like to thank the anonymous referee whose thoughtful comments and corrections greatly improved this paper. 
 \vskip.2cm
 {\bf Declarations:} The author has no competing interests to declare that are relevant to the content of this article.  No data was used in this study.

\section{Preliminaries}\label{sec:Prelim}
\subsection{Doubling and connectedness properties}\label{sec:Doubling}
Given a metric measure space $(X,d,\mu)$, we say that the measure $\mu$ is \emph{doubling} if there exists a constant $C_\mu\ge 1$ such that 
\[
0<\mu(B(x,2r))\le C_\mu\mu(B(x,r))<\infty
\]
for all $x\in X$ and $r>0$.  We say that $(X,d,\mu)$ is a \emph{doubling metric measure space} if $\mu$ is a doubling measure.  By iterating this condition, there are constants $C_Q\ge 1$ and $Q\ge 1$, depending only on $C_\mu$, such that 
\begin{equation}\label{eq:Lower Mass Bound}
\frac{\mu(B(y,r))}{\mu(B(x,R))}\ge C_Q^{-1}\left(\frac{r}{R}\right)^Q
\end{equation}
for all $x\in X$, $0<r\le R$, and $y\in B(x,R)$.  We say that $\mu$ is \emph{reverse doubling} if the opposite relationship holds, that is, if there exist constants $C_\sigma\ge 1$ and $\sigma>0$ such that for all $x\in X$, $0<r\le R$, and $y\in B(x,R)$, we have
\begin{equation}\label{eq:Upper Mass bound}
    \frac{\mu(B(y,r))}{\mu(B(x,R))}\le C_\sigma\left(\frac{r}{R}\right)^\sigma.
\end{equation}
The reverse doubling condition holds if $\mu$ is doubling and $X$ is connected, see for example \cite[Corollary~3.8]{BB}, with constants $\sigma$ and $C_\sigma$ depending only on $C_\mu$.  Throughout this paper, we let $C$ denote a constant which depends, unless otherewise stated, only on the structural constants, such as the doubling constant for example.  Its specific value is not of interest, and may vary with each occurrence, even within the same line.  Given quantities $A$ and $B$, we use the notation $A\simeq B$ to mean that there exists such a constant $C\ge 1$ such that $C^{-1}A\le B\le CA$.  Likewise, we use $\lesssim$ and $\gtrsim$ if the left and right inequalities hold, respectively.

As a consequence of the doubling property, we have the following covering lemma, see for example \cite[Appendix~B7]{Gromov} and \cite[Section~3.3]{HKST}.  This lemma holds under the weaker assumption that $(X,d)$ is \emph{metric doubling}.  A metric space is metric doubling if there exists a constant $C\ge 1$ such that any ball in the space can be covered by at most $C$ balls of half the radius, see \cite{H,HKST} for example.  As doubling metric measure spaces are necessarily metric doubling, we state this lemma here for such spaces.

\begin{lem}\label{lem:BoundedOverlapCover}
Let $(X,d,\mu)$ be a doubling metric measure space.  For any $\eps>0$, there exists a countable cover $\{B_i\}_i$ of $X$ by balls with $\rad(B_i)=\eps$ such that the collection $\{\frac{1}{5}B_i\}_i$ is pairwise disjoint.  Moreover, $\{B_i\}_i$ has bounded overlap in the sense that for all $K>0$, there exists a constant $C$ depending only on $K$ and $C_\mu$ such that 
\[
\sum_i\chi_{KB_i}\le C.
\]    
\end{lem}

Here and throughout the paper, we denote by $\rad(B)$ the radius of a ball $B$.  We associate to this cover the following \emph{Lipschitz partition of unity}, see \cite[Appendix~B7]{Gromov}:
\begin{lem}\label{lem:Lipschitz POU}
    Let $(X,d,\mu)$ be a doubling metric measure space, and let $\eps>0$ and $\{B_i\}_i$ be as in Lemma~\ref{lem:BoundedOverlapCover}.  Then there exists a constant $C\ge 1$ such that for each $i$, there exists a $C/\eps$-Lipschitz function $\phii_i^\eps$ satisfying $0\le\phii_i^\eps\le 1$, $\supt(\phii_i^\eps)\subset 2B_i$, and $\sum_i\phii_i^\eps\equiv 1$.  Here the constant $C$ depends only on $C_\mu$. 
\end{lem}
These preceding lemmas will be used in Section~\ref{sec:Existence of Minimizers} to construct a \emph{discrete convolution} of a locally integrable function $f$: 
\[
f_\eps=\sum_if_{B_i}\phii_i^\eps.
\]
Here and throughout the paper, we denote 
\[
f_B=\fint_Bf\,d\mu=\frac{1}{\mu(B)}\int_B f\,d\mu.
\]
At times, we will assume that $X$ satisfies the following connectedness property:

\begin{defn}\label{def:LLC}
    We say that a metric space $(X,d)$ is \emph{linearly locally connected} (LLC) if there exists a constant $C_L\ge 1$ such that for all $x\in X$ and $r>0$, the following conditions hold:
    \begin{enumerate}
        \item Any two points in $B(x,r)$ can be joined by a connected set in $B(x,C_Lr)$.
        \item Any two points in $X\setminus\overline B(x,r)$ can be joined by a connected set in $X\setminus \overline B(x,r/C_L).$
    \end{enumerate}
    We say that $X$ is LLC-1 if the first condition holds and LLC-2 if the second condition holds.
\end{defn}

\subsection{Measure-theoretic boundary, Minkowski content, and Hausdorff measure}\label{sec:Boundaries and Measures}

\begin{defn}
    Let $E\subset X$.  We define the \emph{regularized boundary} of $E$, denoted $\partial^+ E$, to be the collection of points $x\in X$ such that 
    \[
    \mu(B(x,r)\cap E)>0\quad\text{and}\quad\mu(B(x,r)\setminus E)>0
    \]
    for all $r>0$.  We define the \emph{measure-theoretic boundary} of $E$, denoted $\partial^*E$, to be collection of points $x\in X$ such that 
    \[
    \limsup_{r\to 0^+}\frac{\mu(B(x,r)\cap E)}{\mu(B(x,r))}>0\quad\text{and}\quad\limsup_{r\to0^+}\frac{\mu(B(x,r)\setminus E)}{\mu(B(x,r))}>0.
    \]
    We note that $\partial^* E\subset \partial^+ E\subset\partial E$, with $\partial E$ being the topological boundary of $E$. 
\end{defn}

With the following lemma, we show that any measurable set can be modified on a set of measure zero so that its regularized boundary coincides with its topological boundary. 

\begin{lem}\label{lem:Set Representative}
Let $(X,d,\mu)$ be a doubling metric measure space.  If $E\subset X$ is measurable, then there exists  $\wtil E\subset X$ such that $\mu(E\Delta\wtil E)=0$ and $\partial^+E=\partial^+\wtil E=\partial\wtil E$.
\end{lem}

\begin{proof}
    We define the following sets:
    \begin{align*}
    &E_-:=\{x\in E:\exists r>0\st \mu(B(x,r)\cap E)=0\}\\
    &E_+:=\{x\in X\setminus E:\exists r>0\st\mu(B(x,r)\setminus E)=0\}.
    \end{align*}
    Since $\mu$ is doubling, and so balls have positive $\mu$-measure, it follows that $E_-\cup E_+\subset\partial E$.
    Let $\wtil E=(E\setminus E_-)\cup E_+$.  It follows that $\mu(E_-)=0=\mu(E_+)$ by the Lebesgue differentiation theorem, and so $\mu(E\Delta\wtil E)=0$.  As such, we have that $\partial^+E=\partial^+\wtil E$ by the definition of the regularized boundary. Since $\partial^+\wtil E\subset\partial\wtil E$, it remains to show the opposite inclusion. 

    Let $z\in \partial\wtil E$, and let $r>0$.  Then there exists $x\in B(z,r/2)\cap\wtil E$.  If $x\in E\setminus E_-$, then for all $\rho>0$, we have that 
    \[
    \mu(B(x,\rho)\cap \wtil E)=\mu(B(x,\rho)\cap E)>0,
    \]
    since $\mu(E\Delta\wtil E)=0$. Hence, $\mu(B(z,r)\cap \wtil E)\ge\mu(B(x,r/2)\cap\wtil E)>0$.  If $x\in E_+$, then there exists $r_x>0$ such that $\mu(B(x,r_x)\setminus \wtil E)=\mu(B(x,r_x)\setminus E)=0$, and so we necessarily have that 
    \[
    \mu(B(z,r)\cap\wtil E)>0. 
    \]
    Hence, for all $r>0$, we have that $\mu(B(z,r)\cap\wtil E)>0$.  
    
    Likewise, since $z\in\partial\wtil E$, there exists $y\in B(z,r/2)\setminus\wtil E$.  If $y\in E$, then we have that $y\in E_-$, in which case there exists $r_y>0$ such that $\mu(B(y,r_y)\cap\wtil E)=\mu(B(y,r_y)\cap E)=0$. Then necessarily we have that $\mu(B(z,r)\setminus\wtil E)>0$.  If $y\in X\setminus E$, then $y\in X\setminus E_+$, and so 
    \[
    \mu(B(z,r)\setminus\wtil E)\ge\mu(B(y,r/2)\setminus \wtil E)=\mu(B(y,r/2)\setminus E)>0.
    \]
    Therefore, for all $r>0$, we have $\mu(B(z,r)\setminus\wtil E)>0$, and so $z\in\partial^+\wtil E$.    
    % Since $E_-=(X\setminus E)_+$ and $E_+=(X\setminus E)_-$, it follows from symmetry that $\mu(B(z,r)\setminus\wtil E)>0$ for all $r>0$, and so we have that $z\in\partial^+\wtil E$, completing the proof. 
\end{proof}

\begin{defn}
    Let $E\subset X$, let $t>0$, and let $r>0$. We define the \emph{codimension $t$ Minkowski $r$-content} of $E$ by 
    \[
    \M^{-t}_r(E):=\inf\left\{r^{-t}\sum_i\mu(B(x_i,r)):E\subset\bigcup_i B(x_i,r),\,x_i\in E\right\},
    \]
    and we define the \emph{codimension $t$ upper and lower Minkowski contents} of $E$, respectively, by
    \[
    \overline\M^{-t}(E)=\limsup_{r\to 0^+}\M^{-t}_r(E)\quad\text{ and }\quad \underline\M^{-t}(E)=\liminf_{r\to 0^+}\M^{-t}_r(E).
    \]
    The \emph{upper and lower Minkowski codimensions} of $E$ are defined, respectively, by
    \[
    \overline\codim_\M(E)=\sup\{t>0:\overline\M^{-t}(E)<\infty\}\quad\text{and}\quad\underline\codim_\M(E)=\sup\{t>0:\underline\M^{-t}(E)<\infty\}.
    \]
    Note that $\underline\M^{-t}(E)\le\overline\M^{-t}(E)$, and so $\overline\codim_\M(E)\le\underline\codim_\M(E)$.
    We define the \emph{codimension $t$ Hausdorff $r$-content} of $E$ to be 
    \[
    \Ha^{-t}_r(E):=\inf\left\{\sum_i\frac{\mu(B_i)}{\rad(B_i)^{t}}:E\subset\bigcup_i B_i,\,\rad(B_i)\le r\right\},
    \]
    and likewise, we define the \emph{codimension $t$ Hausdorff measure} of $E$ by 
    \[
    \Ha^{-t}(E)=\lim_{r\to 0^+}\Ha^{-t}_r(E).
    \]
    The \emph{Hausdorff codimension} of $E$ is defined by 
    \[
    \codim_\Ha(E)=\sup\{t>0:\Ha^{-t}(E)<\infty\}.
    \]    
    At times we will be considering Hausdorff contents and measures with respect to different references measures.  In these cases we will include the reference measure as a subscript in the notation ($\Ha^{-t}_{\mu,r}(E)$ and $\Ha^{-t}_\mu(E)$, for example).  However, if the reference measure is clear from context, we will drop the reference to simplify the notation.
\end{defn}

\subsection{Newton-Sobolev, BV, and Besov classes}\label{sec:Sobolev BV Besov}
Let $1\le p<\infty$.  Given a family $\Gamma$ of non-constant, compact, rectifiable curves, we define the $p$-modulus of $\Gamma$ by 
\[
\Mod_p(\Gamma)=\inf_\rho\int_X\rho^pd\mu,
\]
where the infimum is taken over all Borel functions $\rho:X\to[0,\infty]$ such that $\int_\gamma \rho\,ds\ge 1$ for all $\gamma\in\Gamma$.  We refer the interested reader to \cite{HKST} for more on modulus of curve families.  Given a function $u:X\to\overline\R$, we say that a Borel function $g:X\to[0,\infty]$ is an \emph{upper gradient} of $u$ if the following holds for all non-constant, compact, rectifiable curves $\gamma:[a,b]\to X$:
\[
|u(x)-u(y)|\le\int_\gamma g\,ds,
\]
whenever $u(x)$ and $u(y)$ are both finite, and $\int_\gamma g\,ds=\infty$ otherwise.  Here $x$ and $y$ denote the endpoints of $\gamma$.  Upper gradients were first defined in \cite{HK}.  We say that $g$ is a \emph{$p$-weak upper gradient} of $u$ if the $p$-modulus of the family of curves where the above inequality fails is zero.

For $1\le p<\infty$, we define $\wtil N^{1,p}(X)$ to be the class of all functions in $L^p(X)$ which have an upper gradient belonging to $L^p(X)$, and we define 
\[
\|u\|_{\wtil N^{1,p}(X)}=\|u\|_{L^p(X)}+\inf_g\|g\|_{L^p(X)},
\]
where the infimum is taken over all upper gradients $g$ of $u$.  We then define an equivalence relation in $\wtil N^{1,p}(X)$ by $u\sim v$ if and only if $\|u-v\|_{\wtil N^{1,p}(X)}=0$.  The \emph{Newton-Sobolev space} $N^{1,p}(X)$ is then defined to be $\wtil N^{1,p}(X)/\sim$, equipped with the norm $\|\cdot\|_{N^{1,p}(X)}:=\|\cdot\|_{\wtil N^{1,p}(X)}$.  One can similarly define $N^{1,p}(\Omega)$ for any open set $\Omega\subset X$.  Each $u\in N^{1,p}(X)$ has a \emph{minimal $p$-weak upper gradient}, denoted $g_u$, which is unique up to sets of measure zero, see for example \cite[Chapter~6]{HKST}.  For more on $p$-modulus, upper gradients, and the Newton-Sobolev class, see \cite{S,BB,H,HKST}.  

Following Miranda Jr.\ \cite{M} we define the \emph{total variation} of a function $u\in L^1_\loc(X)$, by 
\[
\|Du\|(X):=\inf\left\{\liminf_{i\to\infty}\int_X g_{u_i}\,d\mu:N^{1,1}_\loc(X)\ni u_i\to u\text{ in }L^1_\loc(X)\right\},
\]
where $g_{u_i}$ is an upper gradient of $u_i$.  For $u\in L^1(X)$, we say that $u\in BV(X)$, that is, $u$ is of \emph{bounded variation}, if $\|Du\|(X)<\infty$.  Similarly, we can define $\|Du\|(\Omega)$ and $BV(\Omega)$ for open sets $\Omega\subset X$.  For an arbitrary set $A\subset X$, we then set 
\[
\|Du\|(A):=\inf\{\|Du\|(\Omega):A\subset\Omega,\,\Omega\text{ open}\}.
\]
In \cite{M}, it was shown that $\|Du\|(\cdot)$ is a finite Radon measure on $X$ when $u\in BV(X)$.  For a measurable set $E\subset X$, we say that $E$ is a set of \emph{finite perimeter} if $\|D\chi_E\|(X)<\infty$, and we denote the \emph{perimeter of $E$ in $\Omega$} by 
\[
P(E,\Omega):=\|D\chi_E\|(\Omega).
\]
For more on BV functions in the Euclidean and metric settings, see \cite{EG,AFP,AMP,A}.

For $0<s<1$ and $1\le p<\infty$, we define the \emph{Besov energy} of a function $u\in L^1_\loc(X)$ by 
\[
\|u\|^p_{B^s_{p,p}(X)}:=\int_X\int_X\frac{|u(y)-u(x)|^p}{d(x,y)^{sp}\mu(B(x,d(x,y)))}d\mu(y)d\mu(x).
\]
We then define the \emph{Besov space} $B^s_{p,p}(X)$ to be the set of all functions $u\in L^p(X)$ for which this energy is finite.  In $\R^n$, the space $B^s_{p,p}(\R^n)$ corresponds to the fractional Sobolev space $W^{s,p}(\R^n)$.  In \cite{GKS}, it was shown that in doubling metric measure spaces supporting a $(1,p)$-Poincar\'e inequality (see below), $B^\alpha_{p,p}(X)$ coincides with the real interpolation space $(L^p(X),KS^{1,p}(X))_{\alpha,p}$, where $KS^{1,p}(X)$ is the Korevaar-Schoen Sobolev space.   

\subsection{Poincar\'e inequalities and consequences}\label{sec:Poincare inequalities}
Let $1\le p<\infty$.  We say that $(X,d,\mu)$ supports a \emph{$(1,p)$-Poincar\'e inequality} if there exist constants $C\ge 1$ and $\lambda\ge 1$ such that 
\[
\fint_B|u-u_B|d\mu\le C\rad(B)\left(\fint_{\lambda B} g^p\,d\mu\right)^{1/p}
\]
for all balls $B\subset X$ and function-upper gradient pairs $(u,g)$.  If $X$ is a length space supporting a $(1,p)$-Poincar\'e inequality, then we may take $\lambda=1$ \cite{HaKo}.  If $u$ is locally integrable and possesses a $p$-integrable $p$-weak upper gradient in $X$, then the Poincar\'e inequality above holds with $g_u$, the minimal $p$-weak upper gradient of $u$, on the right hand side, see \cite[Proposition~8.1.3]{HKST}.

When $(X,d,\mu)$ is a complete, doubling metric measure space supporting a $(1,1)$-Poincar\'e inequality, Federer's characterization of sets of finite perimeter holds.  That is, $E\subset X$ is such that $P(E,X)<\infty$ if and only if $\Ha^{-1}(\partial^*E)<\infty$ \cite{A,L2}.  In fact, it was shown by Ambrosio \cite{A} (who proved the ``only if'' direction) and Lahti \cite{L3} (who proved the ``if'' direction), that there exists $\gamma>0$, depending only on the doubling and Poincar\'e constants, such that $P(E,X)<\infty$ if and only if $\Ha^{-1}(\Sigma_\gamma E)<\infty$.  Here, $\Sigma_\gamma E$ is the set of all $x\in X$ such that 
\[
\liminf_{r\to 0^+}\frac{\mu(B(x,r)\cap E)}{\mu(B(x,r))}\ge\gamma\quad\text{and}\quad\liminf_{r\to 0^+}\frac{\mu(B(x,r)\setminus E)}{\mu(B(x,r))}\ge\gamma.
\]

Let $1\le p,q<\infty$ and let $0<s<1$.  We say that $X$ supports an \emph{$(s,q,p,p)$-Poincar\'e inequality} if there exist constants $C\ge 1$ and $\lambda\ge 1$ such that 
\[
\left(\fint_B|u-u_B|^qd\mu\right)^{1/q}\le C\rad(B)^s\left(\fint_{\lambda B}\int_{\lambda B}\frac{|u(y)-u(x)|^p}{d(x,y)^{sp}\mu(B(x,d(x,y)))}d\mu(y)d\mu(x)\right)^{1/p}
\]
for all $u\in L^1_\loc(X)$ and balls $B\subset X$.  

The validity of a $(1,p)$-Poincar\'e inequality implies a number of strong geometric properties of $X$, and as a result, it is often assumed as a standing assumption on $X$.  Under mild assumptions on $X$ however, the fractional $(s,q,p,p)$-Poincar\'e inequality always holds:
\begin{lem}\cite[Lemma~2.2]{DLV}
Let $(X,d,\mu)$ be a doubling metric measure space.  Let $0<s<1$ and let $1\le p,q<\infty$ be such that $q\le p$.  Then $X$ supports an $(s,q,p,p)$-Poincar\'e inequality, with $\lambda=1$ and constant $C$ depending only on the $C_\mu$, $s$, $q$, and $p$.
\end{lem}

If we additionally assume that $\mu$ is reverse doubling, then the exponent on the left-hand side can be improved as follows:
\begin{thm}\cite[Theorem~3.4]{DLV}\label{thm:FractionalPoincare}
    Let $\mu$ be reverse doubling.  Let $0<s<1$ and $1\le p<\infty$ be such that $sp<Q$, where $Q\ge 1$ is as in \eqref{eq:Lower Mass Bound}, and let $p^*=Qp/(Q-sp)$.  Then $X$ supports an $(s,p^*,p,p)$-Poincar\'e inequality with $\lambda=2$ and the constant $C$ depending only on $C_\mu$, $s$, $p$, and the constants from \eqref{eq:Lower Mass Bound} and \eqref{eq:Upper Mass bound}.
\end{thm}
By applying this result to the case $p=1$, we obtain the following lemma, the proof of which follows mutatis mutandis from \cite[Lemma~2.2]{KKLS}.  We include the proof for the reader's convenience.

\begin{lem}\label{lem:FracMazya}
    Let $\mu$ be reverse doubling, and let $0<s<1$, and let $Q\ge1$ be as in \eqref{eq:Lower Mass Bound}.  There exists $C\ge 1$ such that if $u\in L^1_\loc(X)$ and a ball $B\subset X$ are such that 
    \begin{equation}\label{eq:MazyaGamma}
    \frac{\mu(\{|u|>0\}\cap B)}{\mu(B)}\le\gamma
    \end{equation}
    for some $0<\gamma<1$, then
    \begin{equation*}
        \left(\fint_{B}|u|^{Q/(Q-s)}d\mu\right)^{(Q-s)/Q}\le C\frac{\rad(B)^s}{1-\gamma^{s/Q}}\fint_{2B}\int_{2B}\frac{|u(x)-u(y)|}{d(x,y)^s\mu(B(x,d(x,y)))}d\mu(y)d\mu(x).
    \end{equation*}
    Here the constant $C$ depends only on $C_\mu$, $s$, and the constants from \eqref{eq:Lower Mass Bound} and \eqref{eq:Upper Mass bound}.
\end{lem}

\begin{proof}
    Let $u\in L^1_\loc(X)$ and let $B\subset X$ be a ball such that \eqref{eq:MazyaGamma} holds for some $0<\gamma<1$.  For $x\in X$, let 
    \[
    g_{u,s,2B}(x):=\int_{2B}\frac{|u(y)-u(x)|}{d(x,y)^s\mu(B(x,d(x,y)))}d\mu(y),
    \]
    and let $t=Q/(Q-s)$.
    By Minkowski's inequality and Theorem~\ref{thm:FractionalPoincare}, we have that 
    \begin{align*}
        \left(\fint_B|u|^td\mu\right)^{1/t}&\le\left(\fint_B|u-u_B|^td\mu\right)^{1/t}+|u_B|\lesssim\rad(B)^s\fint_{2B}g_{u,s,2B}\,d\mu+|u_B|
    \end{align*}
    We then have from H\"older's inequality that 
    \begin{align*}
    |u_B|\le\frac{1}{\mu(B)}\int_{B\cap\{|u|>0\}}|u|d\mu&\le\frac{1}{\mu(B)}\left(\int_{B\cap\{|u|>0\}}|u|^td\mu\right)^{1/t}\mu(B\cap\{|u|>0\})^{1-1/t}\\
    &\le\gamma^{1-1/t}\left(\fint_B|u|^td\mu\right)^{1/t}.
    \end{align*}
    Absorbing this term to the left hand side of the previous inequality gives
    \[
    \left(\fint_B|u|^td\mu\right)^{1/t}\lesssim\frac{\rad(B)^s}{1-\gamma^{1-1/t}}\fint_{2B}g_{u,s,2B}\,d\mu,
    \]
    with comparison constant coming from Theorem~\ref{thm:FractionalPoincare}.  Hence, this constant depends only on $C_\mu$, $s$, and the constants from \eqref{eq:Lower Mass Bound} and \eqref{eq:Upper Mass bound}.
\end{proof}

\subsection{Hyperbolic fillings}\label{sec:Hyperbolic Filling} 
For a sufficiently regular domain $\Omega\subset\R^n$, $B^{1-\theta/p}_{p,p}(\partial\Omega)$ arises naturally as the trace space of $W^{1,p}(\Omega)$, where $n-\theta$ is the dimension of $\partial\Omega$, see for example \cite{Gal,JW}.  Such a relationship also holds between the Newton-Sobolev class $N^{1,p}(\Omega,\mu)$ and $B^{1-\theta/p}_{p,p}(\partial\Omega,\nu)$ when $\Omega\subset X$ is a uniform domain, and $(X,d,\mu)$ is a doubling metric measure space supporting a $(1,p)$-Poincar\'e inequality \cite{Ma}.  Here $\nu$ is a measure on $\partial\Omega$ which is codimension $\theta$ Ahlfors regular with respect to $\mu$.  That is, there exists a constant $C\ge 1$ such that for all $x\in\partial\Omega$, and $0<r<2\diam(\Omega)$, 
\[
C^{-1}\nu(B(x,r)\cap\partial\Omega)\le\frac{\mu(B(x,r)\cap\Omega)}{(\rad(B))^{\theta}}\le C\nu(B(x,r)\cap\partial\Omega)).
\]
Recently, it was shown by Bj\"orn, Bj\"orn, and Shanmugalingam \cite{BBS} that every compact, doubling metric measure space $(Z,d,\nu)$ arises as the boundary of a uniform space $\Omega$, where the above codimensional relationship between $\nu$ and the measure on $\Omega$ is satisfied.  It was also shown that the Besov spaces on $Z$ arise as traces of the Newton-Sobolev classes on $\Omega$.  The space $\Omega$ was constructed using the hyperbolic filling technique, first introduced by Bonk and Kleiner \cite{BK} and Bourdon and Pajot \cite{BP}.  We outline the construction from \cite{BBS} as follows:

Let $(Z,d,\nu)$ be a compact, doubling metric measure space.  Fix $\alpha>1$, $\tau>1$, and $x_0\in Z$.  For each $n\in\N\cup\{0\}$, choose a maximally $\alpha^{-n}$-separated set $A_n\subset Z$. By scaling the metric if necessary, we may assume that $\diam (Z)<1$, and so $A_0:=\{x_0\}$ for some $x_0\in Z$.  We define the vertex set 
\[
V=\bigcup_{n=0}^\infty\{(x,n):x\in A_n\}.
\]
We then define the edge relationship $\sim$ between vertices by 
$(x,n)\sim(y,m)$ if and only if either $n=m$ and $ B_Z(x,\tau\alpha^{-n})\cap B_Z(y,\tau\alpha^{-m})\ne\varnothing$, or $|n-m|=1$ and $B_Z(x,\alpha^{-n})\cap B_Z(y,\alpha^{-m})\ne\varnothing$.  Thus, vertices $(x,n)$ and $(y,m)$ can only be joined by an edge if $|n-m|\le 1$.  Here  $B_Z$ means that we are considering the ball in $Z$.  We then turn the combinatorial graph given by $V$ and this edge relationship into a metric graph $X$ by assigning a unit length interval to each edge. 
 We define the \emph{hyperbolic filling} of $(Z,d,\nu)$ to be the metric space $(X,d_X)$, where $d_X$ is the path metric on $X$.  

For $\eps>0$, we define the uniformized metric $d_\eps$ on $X$ by
\begin{equation}\label{eq:Unif Metric}
d_\eps(x,y):=\inf_\gamma\int_\gamma e^{-\eps d_X(\cdot, v_0)}ds,
\end{equation}
where $v_0=(x_0,0)$, and the infimum is over all curves in $X$ with endpoints $x$ and $y$. By choosing $\eps=\log\alpha$, it was shown in \cite{BBS} that $(X,d_\eps)$ is a bounded uniform domain.  Furthermore, denoting $\overline X_\eps$ to be the completion of $X$ with respect to $d_\eps$ and $\partial_\eps X:=\overline X_\eps\setminus X$, it was shown that $Z$ is bi-Lipschitz equivalent to $\partial_\eps X$, with constants depending only on $\alpha$ and $\tau$.  Moreover, $\overline X_\eps$ is geodesic.  That is, every two points $x,y\in\overline X_\eps$ can be joined by a curve $\gamma$ such that $\ell(\gamma)=d_\eps (x,y)$. 

Let $\beta>0$.  For each vertex $v=(z,n)\in V$, we define the weight 
\[
\hat\mu_\beta(\{v\})=e^{-\beta n}\nu(B_Z(z,\alpha^{-n})),
\]
and for $A\subset X$, we define the measure $\mu_\beta$ by 
\[
\mu_\beta(A)=\sum_{v\in V}\sum_{w\sim v}(\hat\mu_\beta(\{v\})+\hat\mu_\beta(\{w\}))\Ha^1(A\cap[v,w]),
\]
where $[v,w]$ denotes the edge joining $v$ to $w$.  For $\eps=\log\alpha$ and $\beta>0$, it was shown in \cite{BBS} that the metric measure space $(X,d_\eps,\mu_\beta)$ is doubling and supports a $(1,1)$-Poincar\'e inequality, as does $(\overline X_\eps,d_\eps,\mu_\beta)$, with constants depending only on $\alpha$, $\tau$, $\beta$, and the doubling constant of $\nu$. Furthermore, the following codimensional relationship holds between $\nu$ and $\mu_\beta$:
\begin{equation}\label{eq:HypFillCodim}
    \frac{\mu_\beta(B_\eps(\zeta,r))}{r^{\beta/\eps}}\simeq\nu(B_Z(\zeta,r))
\end{equation}
for all $\zeta\in\partial_\eps X$ and $0<r\le 2\diam X_\eps$.  Here the comparison constants also depend only on $\alpha$, $\beta$, $\tau$, and the doubling constant of $\nu$. It was then shown in \cite{BBS} that for $1\le p<\infty$, there exist bounded trace and extension operators between $N^{1,p}(\overline X_\eps,\mu_\beta)$ and $B^{1-\beta/(\eps p)}_{p,p}(Z,\nu)$.  In particular, the following theorem holds, which we will utilize in Section~\ref{sec:Size of Boundary}:

\begin{thm}\cite[Theorem~12.1]{BBS}\label{thm:HypFillThm}
Let $1\le p<\infty$ and let $0<s<1$.  Let $\eps=\log\alpha$ and choose $\beta>0$ so that $s=1-\beta/(p\eps)$.  Then for each $f\in B^s_{p,p}(Z,\nu)$, there exists $u_f\in N^{1,p}(\overline X_\eps,\mu_\beta)$ such that $u_f=f$ $\nu$-a.e.\ on $Z$, and 
\begin{equation*}
    \int_{\overline X_\eps}g^p_{u_f}d\mu_\beta\lesssim\|f\|^p_{B^s_{p,p}(Z,\nu)}\quad\text{and}\quad\int_{\overline X_\eps}|u_f|^pd\mu_\beta\lesssim\int_Z|f|^pd\nu,
\end{equation*}
where $g_{u_f}$ is the minimal $p$-weak upper gradient of $u_f$, and the comparison constants depend only on $s$, $p$, $\alpha$, $\tau$, and the doubling constant of $\nu$.  Furthermore, for $\nu$-a.e.\ $z\in Z$, we have that 
\[
\lim_{r\to0^+}\fint_{B_\eps(z,r)}|u_f-f(z)|^pd\mu_\beta=0.
\]
That is, $\nu$-a.e.\ $z\in Z$ is an $L^p(\mu_\beta)$-Lebesgue point of $u_f$.
\end{thm}

By \eqref{eq:HypFillCodim}, we have the following lemmas, which relate $\nu$ and codimensional Hausdorff measures on $Z$ to codimensional Haudsorff measures on $\overline X_\eps$:

\begin{lem}\label{lem:ContentCompare}
Let $\eps=\log\alpha$ and $\beta>0$.  Let $\overline X_\eps$ be the completion with respect to $d_\eps$ of the uniformized hyperbolic filling $(X,d_\eps,\mu_\beta)$ of $(Z,d_Z,\nu)$.  If $z\in Z$, $\delta>0$, and $A\subset B_Z(z,r)\cap Z$, then 
\[
\nu(A)\lesssim\Ha^{-\beta/\eps}_{\mu_\beta,\,\delta}(A).
\]
\end{lem}

\begin{proof}
    Let $\{B_k\}_k$ be a sequence of balls $B_k\subset\overline X_\eps$ with $\rad(B_k)\le \delta$ such that $A\subset\bigcup_k B_k$.  Without loss of generality, we may assume that $B_k\cap A\ne\varnothing$ for all $k$. Then, for each $k$, there exists $B_k'$ such that $B_k\subset B_k'$, $B_k'$ is centered at a point in $Z$, and $\rad(B_k)\simeq\rad(B_k')$.  We then have that 
    \[
    \sum_k\frac{\mu_\beta (B_k)}{\rad(B_k)^{\beta/\eps}}\simeq\sum_k\frac{\mu_\beta (B_k')}{\rad(B_k')^{\beta/\eps}}\simeq\sum_k\nu(B_k')\ge \nu(A).
    \]
    Here we have used the doubling property of $\mu_\beta$, and bi-Lipschitz equivalence of $d_\eps$ and $d_Z$ on $Z$ in conjunction with \eqref{eq:HypFillCodim}.  This completes the proof, as the cover $\{B_k\}_k$ is arbitrary. 
    \end{proof}

\begin{lem}\label{lem:CodimComp}
Let $\eps=\log\alpha$ and $0<\beta\le\eps$.  Let $\overline X_\eps$ be the completion with respect to $d_\eps$ of the uniformized hyperbolic filling $(X,d_\eps,\mu_\beta)$ of $(Z,d_Z,\nu)$.  Then 
\[
\Ha^{-1}_{\mu_\beta}|_Z\simeq\Ha^{-(1-\beta/\eps)}_\nu.
\]
\end{lem}

\begin{proof}
    The result follows directly from \eqref{eq:HypFillCodim}.  Note that by by the bi-Lipschitz equivalence of $d_\eps$ and $d_Z$ on $Z$, we can equivalently define $\Ha^{-(1-\beta/\eps)}_\nu$ with respect to $d_\eps$.  
\end{proof}

\subsection{Fractional $s$-perimeter and minimizers of the $\J_\Omega^s$ functional}\label{sec:Fractional Perimeter} 
% Let $(X,d,\mu)$ be a doubling metric measure space, and let $0<s<1$.  For measurable sets $A,B\subset X$, we define 
%  \[
%  L_s(A,B):=\int_A\int_B\frac{1}{d(x,y)^s\mu(B(x,d(x,y)))}d\mu(y)d\mu(x).
%  \] 
 It was shown in \cite{D} that for a measurable set $E\subset\R^n$, the following formula holds:
 \[
 \lim_{s\to 1^-}(1-s)\|\chi_E\|_{B^s_{1,1}(\R^n)}=C_nP(E,\R^n),
 \]
 where $C_n$ is a dimensional constant and $P(E,\R^n)=\|D\chi_E\|(\R^n)$ is the perimeter of $E$ in $\R^n$, see Section~\ref{sec:Sobolev BV Besov}.  More recently, it was shown in \cite{DS} that in a complete, doubling metric measure space supporting a $(1,1)$-Poincar\'e inequality, the following holds for measurable sets $E\subset X$:
 \[
 P(E,X)\lesssim\liminf_{s\to 1^-}(1-s)\|\chi_E\|_{B^s_{1,1}(X)}\le\limsup_{s\to 1^-}(1-s)\|\chi_E\|_{B^s_{1,1}(X)}\lesssim P(E,X).
 \]
 In this manner, when $p=1$, the Besov energy of the characteristic function of a set $E$ recovers the perimeter of $E$ under suitable rescaling as $s\to 1^-$.  For $0<s<1$, we therefore define the \emph{$s$-perimeter} of a measurable set $E\subset X$, by 
 \[
 P_s(E,X):=\|\chi_E\|_{B^s_{1,1}(X)}=\int_X\int_X\frac{|\chi_E(x)-\chi_E(y)|}{d(x,y)^s\mu(B(x,d(x,y)))}d\mu(y)d\mu(x).
 \]
 If $\mu$ is doubling, then we have that 
 \begin{equation}\label{eq:FracPerKernel}
     P_s(E,X)\simeq\int_{E}\int_{X\setminus E}\frac{1}{d(x,y)^s\mu(B(x,d(x,y)))}d\mu(y)d\mu(x).
 \end{equation}
 The $s$-perimeter generates the following notion of codimension, as originally defined in \cite{V}:
\begin{defn}
    Let $E\subset X$ be measurable.  We define the \emph{fractional codimension} of $E$ by 
    \[
    \codim_F(E)=\sup\{0<s<1:P_s(E,X)<\infty\}.
    \]
\end{defn}
\noindent In Section~\ref{sec:Size of Boundary} and \ref{sec:codim Examples}, we will study the relationship between the fractional codimension of a set and the Minkowski and Hausdorff codimensions of its boundary. 

For $x,y\in X$, we define the kernel 
\[
K_s(x,y):=\frac{2}{d(x,y)^s[\mu(B(x,d(x,y)))+\mu(B(y,d(x,y)))]},
\]
and for measurable sets $A,B\subset X$, we adopt the following notation:
\[
L_s(A,B):=\int_A\int_BK_s(x,y)d\mu(y)d\mu(x).
\]
By symmetry of the kernel $K_s$, we have that 
\[
L_s(A,B)=L_s(B,A),
\]
and if $\mu$ is doubling, then 
\begin{equation}\label{eq:LsKernel}
L_s(A,B)\simeq\int_A\int_B\frac{1}{d(x,y)^s\mu(B(x,d(x,y)))}d\mu(y)d\mu(x),
\end{equation}
with comparison constants depending only on the doubling constant of $\mu$.

 Given a bounded domain $\Omega\subset X$ with $X\setminus\Omega\ne\varnothing$, and a set $E\subset X$, we define the following functional $\J^s_\Omega$ as introduced in \cite{CRS} in the Euclidean setting:
 \begin{align}\label{eq:JOmegaDef}
     \J_\Omega^s(E):=L_s(E\cap\Omega,X\setminus E)+L_s(E\setminus\Omega,\Omega\setminus E).
 \end{align}
 From \eqref{eq:FracPerKernel}, \eqref{eq:LsKernel}, and \eqref{eq:JOmegaDef}, we see that in doubling metric measure spaces, this functional incorporates much of the $s$-perimeter of $E$, but it excludes the term $L_s(E\setminus\Omega,X\setminus (E\cup\Omega))$ from the $s$-perimeter.  The reason for this exclusion is that this term may be infinite, and in the following minimization problem, candidate sets will be fixed outside of $\Omega$.  Therefore, this exclusion does not affect the minimization process.  
 \begin{defn}\label{def:Fractional Minimizer}
    Let $\Omega\subset X$ be a bounded domain, and let $0<s<1$.  Let $F\subset X$ be a measurable set such that $F\setminus\Omega\ne\varnothing$.  We say that $E\subset X$ is a \emph{minimizer} of $\J_\Omega^s$ if $\J_\Omega^s(E)<\infty$, $E\setminus\Omega=F\setminus\Omega$, and 
    \[
    \J_\Omega^s(E)\le\J_\Omega^s(E')
    \]
    for all $E'\subset X$ such that $E'\setminus\Omega=F\setminus\Omega$.
 \end{defn}
 This problem was introduced in the Euclidean setting in \cite{CRS}, where existence and regularity of minimizers was studied. 
 We will study existence and regularity of minimizers in the doubling metric measure space setting in Sections~\ref{sec:Existence of Minimizers} and \ref{sec:Regularity of Minimizers}.

We conclude this section with a lemma concerning minimizers of the $\J_\Omega^s$ functional, which can be found in \cite[Section~2]{CRS}.  We will use this lemma to study regularity of minimizers in Section~\ref{sec:Regularity of Minimizers}.

\begin{lem}\label{lem:SupSubSoln}
Let $E$ be a minimizer of $\JOm^s$.  If $A\subset (E\cap\Omega),$ then
\begin{equation}\label{eq:SuperSoln}
L_s(A, X\setminus E)\le L_s(E\setminus A, A).
\end{equation}
\end{lem}

\begin{proof}
   Since $E$ is a minimizer, by taking $E\setminus A$ as a competitor, we obtain 
   \begin{align*}
       0\le\JOm^s(E\setminus A)-\JOm^s(E)=L_s((&E\setminus A)\cap\Omega,X\setminus (E\setminus A))+L_s(E\setminus\Omega,\Omega\setminus (E\setminus A))\\
       &-L_s(E\cap\Omega,X\setminus E)-L_s(E\setminus\Omega,\Omega\setminus E).
   \end{align*}
   Since $\JOm^s(E)<\infty$, we know that $L_s(E\cap\Omega,\Omega\setminus E)<\infty$.  We then have that 
   \[
   L_s(E\setminus\Omega,\Omega\setminus(E\setminus A))-L_s(E\setminus\Omega,\Omega\setminus E)=L_s(E\setminus\Omega, A)=L_s((E\setminus A)\setminus\Omega,A).
   \]
   We also have that 
   \[
   L_s((E\setminus A)\cap\Omega,X\setminus (E\setminus A))=L_s((E\setminus A)\cap\Omega, X\setminus E)+L_s((E\setminus A)\cap\Omega,A).
   \]
    Substituting both of these expressions into the previous expression yields
    \begin{align*}
    0&\le L_s((E\setminus A)\cap\Omega, X\setminus E)+L_s((E\setminus A)\cap\Omega,A)+L_s((E\setminus A)\setminus\Omega,A)-L_s(E\cap\Omega,X\setminus E)\\
    &=L_s(E\setminus A,A)+L_s((E\setminus A)\cap\Omega,X\setminus E)-L_s(E\cap\Omega, X\setminus E)\\
    &=L_s(E\setminus A,A)-L_s(A,X\setminus E).
    \end{align*}
    Note that both $L_s(E\cap\Omega, X\setminus E)$ and $L_s((E\setminus A)\cap\Omega, X\setminus E)$ are finite since $\JOm^s(E)<\infty$.  This allows us to obtain the last equality in the above expression.
\end{proof}

\section{On the boundary size of sets of finite $s$-perimeter}\label{sec:Size of Boundary}

In this section, we let $(Z,d,\nu)$ be a metric measure space, with $\nu$ a doubling Borel regular measure.  Our goal here is to explore the relationship between the $s$-perimeter of a subset of $Z$ and the size of its boundary.  Recall the definition of $s$-perimeter: for $E\subset Z$, we define $P_s(E,Z)=\|\chi_E\|_{B^s_{1,1}(Z)}$.  
We begin by proving a sufficient condition, given in terms of the upper Minkowski content of the regularized boundary, guaranteeing that a set has finite $s$-perimeter.  This result was first proved in the Euclidean setting by Visintin, see \cite[Propositions~11 and 13]{V}.  We show here that the same result holds in doubling metric measure spaces satisfying the LLC-1 condition, see Definition~\ref{def:LLC}.

For ease of notation in the following proof, we define the following for measurable sets $A,B\subset Z$:
\begin{equation*}
    J_s(A,B):=\int_A\int_B\frac{1}{d(z,w)^s\nu(B(z,d(z,w)))}d\nu(w)d\nu(z).
\end{equation*}

\begin{proof}[Proof of Theorem~\ref{thm:MinkowskiSufficient}]
    Let $E\subset Z$ be bounded and measurable, and suppose that there exist $t>0$ such that $\overline\M^{-t}(\partial^+E)<\infty$.  By Lemma~\ref{lem:Set Representative} and the fact that the Besov energy does not detect  measure zero changes in sets, we may assume without loss of generality that $\partial^+E=\partial E$. 
    
    Let $\eps>0$, and let $0<s<t$.  By definition, there exists $r_0>0$ such that for all $0<r<r_0$, there exists a countable cover $\{B(x_i,r)\}_{i\in\N}$ of $\partial E$, with each $x_i\in\partial E$, such that    \begin{equation}\label{eq:Mink 2}
        r^{-t}\sum_i\nu(B(x_i,r))<\overline\M^{-t}(\partial E)+\eps.
    \end{equation}
   For each $k\in\N$, let 
    \[
    E_k:=\{z\in E:2^{-k-1}r_0\le\dist (z,Z\setminus E)<2^{-k}r_0\},
    \]
    and let $E_0=E\setminus\left( \bigcup_{k\in\N}E_k\cup\partial E\right).$  Since $\overline\M^{-t}(\partial E)<\infty$ and $t>0$, it follows that $\nu(\partial E)=0$, and so we have that 
   \begin{align}\label{eq:Mink 1}
    \|\chi_E\|_{B^s_{1,1}(Z,\nu)}\simeq J_s(E,Z\setminus E)\simeq J_s(E_0,Z\setminus E)+\sum_{k\in\N}J_s(E_k,Z\setminus E).
   \end{align}
   For each $z\in E_0$, we have that $\dist(z,Z\setminus E)\ge r_0$, and so by doubling we have that 
   \begin{align}\label{eq:Mink 5}
    J_s(E_0,Z\setminus E)&\le\int_{E_0}\int_{Z\setminus B(z_,r_0)}\frac{1}{d(z,w)^s\nu(B(z,d(z,w)))}d\nu(w)d\nu(z)\nonumber\\
    &\le\int_{E_0}\sum_{m=0}^\infty\int_{B(z,2^{m+1}r_0)\setminus B(z,2^mr_0)}\frac{1}{d(z,w)^s\nu(B(z,d(z,w)))}d\nu(w)d\nu(z)\nonumber\\
    &\simeq\int_{E_0}\sum_{m=0}^\infty (2^mr_0)^{-s}d\nu(z)\simeq r_0^{-s}\nu(E_0)<\infty.
   \end{align}
   Here, the finiteness is due to boundedness of $E$.

   For each $k\in\N$, we have by a similar computation that 
   \begin{align*}
       J_s(E_k,Z\setminus E)&\le\int_{E_k}\int_{Z\setminus B(z,2^{-k-1}r_0)}\frac{1}{d(z,w)^s\nu(B(z,d(z,w)))}d\nu(w)d\nu(z)\\
       &=\int_{E_k}\sum_{m=0}^\infty\int_{B(z,2^{-k+m}r_0)\setminus B(z,2^{-k-1+m}r_0)}\frac{1}{d(z,w)^s\nu(B(z,d(z,w)))}d\nu(w)d\nu(z)\\
       &\simeq\frac{\nu(E_k)}{(2^{-k}r_0)^s}.
   \end{align*}
  For each $z\in E_k$, there exists $w\in Z\setminus E$ such that $d(z,w)<2^{-k}r_0$.  By the LLC-1 condition, there exists a connected set in $B(z,C_L2^{-k}r_0)$ joining $z$ to $w$, where $C_L\ge 1$ is the constant from \eqref{def:LLC}.  Since $z\in E_k\subset E$ and $w\in Z\setminus E$, it follows that there exists $y\in B(z,C_L2^{-k}r_0)\cap \partial E$.  Therefore, 
  \[
  E_k\subset\{z\in Z:\dist(z,\partial E)< C_L2^{-k}r_0\}=:\mathcal{N}_{C_L2^{-k}r_0}(\partial E),
  \]
  and so combining this with the previous inequality, we have
  \begin{equation}\label{eq:Mink 4}
  J_s(E_k,Z\setminus E)\lesssim\frac{\nu(\No_{C_L2^{-k}r_0}(\partial E))}{(2^{-k}r_0)^s}.
  \end{equation}
  By \eqref{eq:Mink 2}, there exists a cover $\{B(x_i,2^{-k}r_0)\}_{i\in\N}$ of $\partial E$, with $x_i\in\partial E$, such that 
  \begin{equation}\label{eq:Mink 3}
  (2^{-k}r_0)^{-t}\sum_{i}\nu(B(x_i,2^{-k}r_0))<\overline\M^{-t}(\partial E)+\eps.
  \end{equation}
  We then claim that 
  \[
  \No_{C_L2^{-k}r_0}(\partial E)\subset\bigcup_{i}B(x_i,2C_L2^{-k}r_0).
  \]
Indeed, if $x\in\No_{C_L2^{-k}r_0}(\partial E)$, then there exists $y\in\partial E$ such that $d(x,y)<C_L2^{-k}r_0$.  There also exists $i\in\N$ such that $y\in B(x_i,2^{-k}r_0)$, and so $d(x,x_i)<(C_L+1)2^{-k}r_0\le 2C_L2^{-k}r_0$.  The claim follows.  Therefore, by \eqref{eq:Mink 4}, \eqref{eq:Mink 3}, and the doubling property of $\nu$, we have that 
  \begin{align*}
      J_s(E_k,Z\setminus E)\lesssim (2^{-k}r_0)^{-s}\sum_{i}\nu(B(x_i,2C_L2^{-k}r_0))&\lesssim(2^{-k}r_0)^{-s}\sum_{i}\nu(B(x_i,2^{-k}r_0))\\
      &\le(2^{-k}r_0)^{t-s}(\overline\M^{-t}(\partial E)+\eps).
  \end{align*}
  Since $t>s$, it follows that 
  \[
  \sum_kJ_s(E_k,Z\setminus E)\lesssim r_0^{t-s}(\overline\M^{-t}(\partial E)+\eps)<\infty.
  \]
  Combining this with \eqref{eq:Mink 5} and \eqref{eq:Mink 1} completes the proof.
\end{proof}

\begin{example}\label{ex:Mink LLC}
   We point out that the conclusion of the above theorem may fail if $Z$ does not satisfy the LLC-1 condition.  For example, consider the standard slit disk in $\R^2$, but replace the slit with a copy of the von Koch snowflake curve. Let $Z$ be this modified slit disk, equipped with Euclidean distance and 2-dimensional Lebesgue measure.  Note that $Z$ is not LLC-1.  Let $E\subset Z$ be the points in $Z$ lying below the curve and with positive first coordinate, see Figure~\ref{fig:Snowflake Slit disk}.  By the argument from the proof of \cite[Theorem~1.1]{Lom}, it follows that $\chi_E\not\in B^{2-\log 4/\log 3}_{1,1}(Z)$.  However, $\partial^+E$ consists of the points $\{(0,y):-1<y<0\}$, and so $\overline\M^{-t}(\partial^+E)=0$ for all $0<t<1$.  This occurs because points in $E$ can be arbitrarily close to $Z\setminus E$ without being near $\partial^+E$.  The LLC-1 condition prevents this. 
\end{example}

\begin{figure}[h]
    \centering
\includegraphics[scale=0.91]{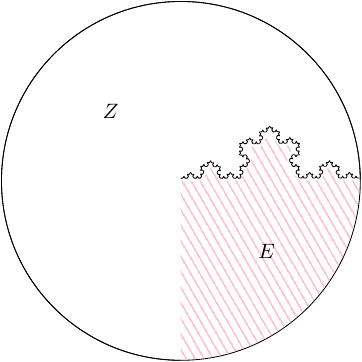}
\caption{Example~\ref{ex:Mink LLC}}\label{fig:Snowflake Slit disk}
\end{figure}

Given a set $E\subset Z$ with finite $s$-perimeter, we now wish to analyze the codimension $s$ Hausdorff measure of the measure-theoretic boundary of $E$.  To do so, we will use the hyperbolic filling, where the presence of a $(1,1)$-Poincar\'e inequality provides us certain potential theoretic tools. 
 In particular, we will use the following lemma, which is a modification of \cite[Theorem~5.9]{HK}. This can be found for example in \cite[Lemma~3.1]{EGKS}, where it was proved in greater generality.  We include the proof of the particular case pertinent to us for the convenience of the reader. 
\begin{lem}\label{lem:Loewner}
Suppose that $(X,d,\mu)$ is a geodesic space, with $\mu$ a doubling measure supporting a $(1,1)$-Poincar\'e inequality, and let $0<t<1$.  Then there exists a constant $C\ge 1$ such that if $r>0$, $E$ and $F$ are subsets of a ball $B(x,r)$, and $\lambda>0$ satisfies
\[
\min\{\Ha^{-t}_r(E),\Ha^{-t}_r(F)\}\ge\lambda\frac{\mu(B(x,r))}{r^t},
\]
then for any function $u\in N^{1,1}(X)$ such that $u=1$ on $E$, $u=0$ on $F$, and each $z\in E\cup F$ such that $z$ is an $L^1(\mu)$-Lebesgue point of $u$, we have that 
\[
\int_{B(x,r)}g_u\,d\mu\ge\frac{\lambda}{C}\frac{\mu(B(x,r))}{r},
\]
where $g_u$ is the minimal $1$-weak upper gradient of $u$.
\end{lem}

\begin{proof}
    By truncation, we may assume that $0\le u\le 1$.  If $u_{B(x,r)}\le 1/2$, then for all $z\in E$, we have that $|u(z)-u_{B(x,r)}|\ge 1/2$. Likewise, if $u_{B(x,r)}>1/2$, then $|u(w)-u_{B(x,r)}|\ge 1/2$ for all $w\in F$.  We consider the first case, with the proof of the second case following analogously.

    Let $z\in E$, and consider a geodesic $[x,z]$ connecting $x$ to $z$.  If $d(x,z)>r/4$, then let $B_0:=B(x,r)$,  $B_1:=B(x,d(x,z)/2)$, and let $B_2:=B(x_2,d(x,z)/2^2)$, where $x_2$ is a point in $[x,z]$ such that $d(x,x_2)=d(x,z)/2$. 
 Inductively, if $B_k:=B(x_k,d(x,z)/2^k)$ is such that $x_k\in [x,z]$ with $d(x_k,x_{k-1})=d(x,z)/2^{k-1}$, then we choose a point $x_{k+1}\in[x,z]$ such that $d(x_{k+1},x_k)=d(x,z)/2^k$ and define $B_{k+1}:=B(x_{k+1},d(x,z)/2^{k+1})$.  If $d(x,z)\le r/4$, then set $B_0=B(x,r)$, and for each $k\in\N$, let $B_k=B(z,2^{-k}r)$.  In either case, we obtain a chain of balls $\{B_k\}_k$ such that for all $k\in\N$, we have that $z\in 3B_k$ and $B_{k+1}\subset 2B_k\subset B(x,r)$.

 Since $z$ is an $L^1(\mu)$-Lebesgue point of $u$, we have by the doubling property of $\mu$ and the $(1,1)$-Poincar\'e inequality, that 
 \begin{align*}
     \frac{1}{2}\le |u(z)-u_{B(x,r)}|\le\sum_{k=0}^\infty|u_{B_{k+1}}-u_{B_k}|\lesssim\sum_{k=0}^\infty\fint_{2B_k}|u-u_{2B_k}|d\mu\lesssim\sum_{k=0}^\infty 2^{-k}r\fint_{2B_k}g_ud\mu.
 \end{align*}
Note that the dilation constant in the $(1,1)$-Poincar\'e inequality is taken to be 1, since $X$ is assumed to be geodesic. Since the series $\sum_{k=0}^\infty 2^{-k(1-t)}$ converges, there exists $k_z\in\N\cup\{0\}$ such that 
 \[
 2^{-k_z(1-t)}\lesssim 2^{-k_z}r\fint_{2B_{k_z}}g_ud\mu.
 \]
 We then have that 
 \[
 \mu(B_{k_z})\lesssim 2^{-k_zt}r\int_{2B_{k_z}}g_ud\mu\lesssim r^{1-t}\rad(B_{k_z})^t\int_{2B_{k_z}}g_ud\mu.
 \]
 Now, $\{3B_{k_z}\}_{z\in E}$ covers $E$, and the balls in this collection have uniformly bounded radii. Thus, by the 5-covering lemma, see \cite[Chapter~3]{HKST}, there exists a pairwise disjoint subcollection $\{3B_i\}_{i\in I\subset\N}$ such that $\{15B_i\}_{i\in I}$ covers $E$.   Furthermore, for each $i\in I$, there exists a countable collection of balls $\{B_{i,j}\}_{j\in I(i)}$ which covers $15 B_i$ and has bounded overlap, such that $B_{i,j}\subset 30 B_i$ and $\rad(B_{i,j})=\rad(B_i)\le r$ for each $j\in I(i)$, see \cite[Lemma~2.6]{EGKS}.  Hence, by the doubling property of $\mu$ and the previous inequality, we have that
 \begin{align*}
     \Ha^{-t}_r(E)&\le\sum_{i\in I}\sum_{j\in I(i)}\frac{\mu(B_{i,j})}{\rad(B_{i,j})^t}\\
    &\lesssim\sum_{i\in I}\frac{\mu(30B_i)}{\rad (B_i)^t}\lesssim\sum_{i\in I}\frac{\mu(B_i)}{\rad(B_i)^t}\lesssim\sum_{i\in I}r^{1-t}\int_{2B_i}g_ud\mu\le r^{1-t}\int_{B(x,r)}g_ud\mu.
 \end{align*}
 In the last inequality, we have used the fact that $2B_i\subset B(x,r)$ and disjointness of $\{2B_i\}_{i\in I}$.  By hypothesis, it then follows that
 \[ \lambda\frac{\mu(B(x,r))}{r}\le\frac{1}{r^{1-t}}\Ha^{-t}_r(E)\lesssim\int_{B(x,r)}g_ud\mu,
 \]
 which completes the proof.
 \end{proof}

We now prove Theorem~\ref{thm:CodimHausZero}.

\begin{proof}[Proof of Theorem~\ref{thm:CodimHausZero}]
     Fix $0<s<1$ and let $E\subset Z$ be such that $\chi_E\in B^s_{1,1}(Z,\nu)$.  Let $(X,d_X)$ be the hyperbolic filling of $(Z,d,\nu)$, constructed with parameters $\alpha$, $\tau>1$.  Let $\eps=\log\alpha$, and choose $\beta>0$ such that $\beta/\eps=1-s$.  Consider the metric $d_\eps$ as given by \eqref{eq:Unif Metric}, and let $\overline X_\eps$ be the completion with respect to $d_\eps$ of the uniformized hyperbolic filling $(X,d_\eps,\mu_\beta)$ of $(Z,d,\nu)$, as constructed in Section~\ref{sec:Hyperbolic Filling}.  Let $u:=u_{\chi_E}\in N^{1,1}(\overline X_\eps,\mu_\beta)$ be the extension of $\chi_E$ given by Theorem~\ref{thm:HypFillThm}.  By Theorem~\ref{thm:HypFillThm}, we have that $u(z)=\chi_E(z)$ for $\nu$-a.e.\ $z\in Z$, and $\nu$-a.e.\ $z\in Z$ is an $L^1(\mu_\beta)$-Lebesgue point of $u$.  Let $E_1:=\{u=1\}\cap Z\setminus N_u$ and $E_0:=\{u=0\}\cap Z\setminus N_u$, where $N_u$ is the set of points which are not $L^1(\mu_\beta)$-Lebesgue points of $u$.   

    Fix $\delta>0$ and $\eta>0$.  For each $k\in\N$ let $\partial^*_k E$ denote the set of points $z\in Z$ such that 
    \[
    \limsup_{r\to 0}\frac{\nu(B_Z(z,r)\cap E)}{\nu(B_Z(z,r))}\ge 2^{-k}\quad\text{ and }\quad\limsup_{r\to 0}\frac{\nu(B_Z(z,r)\setminus E)}{\nu(B_Z(z,r))}\ge 2^{-k}.
    \]
    Recall that we use the notation $B_Z$ to indicate that these are balls with respect to $(Z,d,\nu)$. 
 Letting $N_\delta\in\N$ be the smallest positive integer such that $1/N_\delta<\delta$.  We then have that 
    \begin{equation}\label{eq:Meas Theoretic Boundary Layers}    \partial^*E=\bigcup_{k=N_\delta}^\infty \partial^*_k E,
    \end{equation}
where $\partial^*E$ is the measure-theoretic boundary of $E$ with respect to $(Z,d,\nu)$.

    Fix $k\in\N$ with $k\ge N_\delta$. 
 We note, by the doubling property of $\nu$ and the codimensional relationship \eqref{eq:HypFillCodim} between $\nu$ and $\mu_\beta$, that for all $0<t<2\diam(X_\eps)$, we have  
 \[
 \mu_\beta(\{x\in\overline X_\eps:\dist(x,\partial_\eps X)\le t\})\lesssim t^{\beta/\eps}\nu(Z).
 \]
   Therefore, since $u\in N^{1,1}(\overline X_\eps,\mu_\beta)$, there exists $j_k\in\N$ such that 
    \begin{equation}\label{eq:EtaSmall}
    \int_{X_{j_k}}g_u\,d\mu_\beta\le 2^{-2k}\eta,
    \end{equation}
    where $X_{j_k}:=\{x\in\overline X_\eps:\dist(x,\partial_\eps X)\le 1/j_k\}$ and $g_u$ is the minimal $1$-weak upper gradient of $u$ in $\overline X_\eps$.  Recall from Section~\ref{sec:Hyperbolic Filling} that $(Z,d_\eps)$ is bi-Lipschitz equivalent to $(Z,d)$. 
 Therefore, there exists a constant $C_0\ge 1$ depending only on $\alpha$, $\tau$, and the doubling constant of $\nu$, such that for all $z\in Z$ and $r>0$, 
 \begin{equation}\label{eq:Engulf}
 B_Z(z,r)\subset B_\eps(z,C_0r)\cap Z.
 \end{equation}
 Since $u=\chi_E$ $\nu$-a.e.\ in $Z$ and $\nu$-a.e.\ point in $Z$ is an $L^1(\mu_\beta)$-Lebesgue point of $u$, it follows that for each $\zeta\in \partial^*_k E$, there exists $0<r_{\zeta,k}<\min\{1/(C_0j_k),\delta/(5C_0)\}$ such that 
    \begin{align*}
        \min\{\nu(B_Z(\zeta, r_{\zeta,k})\cap E_1),\nu(B_Z(\zeta,r_{\zeta,k})\cap E_0)\}&=\min\{\nu(B_Z(\zeta, r_{\zeta,k})\cap E),\nu(B_Z(\zeta,r_{\zeta,k})\setminus E)\}\\
        &\ge\frac{1}{2^k}\nu(B_Z(\zeta,r_{\zeta,k}))\simeq\frac{1}{2^k}\frac{\mu_\beta(B_\eps(\zeta,r_{\zeta,k}))}{r_{\zeta,k}^{1-s}}.
    \end{align*}
    Here we have used \eqref{eq:HypFillCodim} and the choice $\beta/\eps=1-s$ in the last inequality. By Lemma~\ref{lem:ContentCompare}, \eqref{eq:Engulf}, and the doubling property of $\mu_\beta$, we have that 
    \begin{align*}
    \min &\left\{\Ha^{-(1-s)}_{\mu_\beta,\,C_0r_{\zeta,k}}\left(B_\eps(\zeta, C_0r_{\zeta,k})\cap E_1\right),\Ha^{-(1-s)}_{\mu_\beta,\,C_0r_{\zeta,k}}(B_\eps(\zeta, C_0r_{\zeta,k})\cap E_0)\right\}\\
    &\gtrsim\min\{\nu(B_\eps(\zeta, C_0r_{\zeta,k})\cap E_1),\nu(B_\eps(\zeta,C_0r_{\zeta,k})\cap E_0)\}\\
    &\ge\min\{\nu(B_Z(\zeta, r_{\zeta,k})\cap E_1),\nu(B_Z(\zeta,r_{\zeta,k})\cap E_0)\}\gtrsim\frac{1}{2^k}\frac{\mu_\beta(B_\eps(\zeta,r_{\zeta,k}))}{r_{\zeta,k}^{1-s}}\simeq\frac{1}{2^k}\frac{\mu_\beta(B_\eps(\zeta,C_0r_{\zeta,k}))}{(C_0r_{\zeta,k})^{1-s}}.
    \end{align*}
    Note that $(\overline X_\eps,d_\eps,\mu_\beta)$ is geodesic, where $\mu_\beta$ is doubling and supports a $(1,1)$-Poincar\'e inequality, see Section~\ref{sec:Hyperbolic Filling}.  Hence, we apply Lemma~\ref{lem:Loewner} to obtain
    \[    \int_{B_\eps(\zeta,C_0r_{\zeta,k})}g_u\,d\mu_\beta\gtrsim\frac{1}{2^k}\frac{\mu_\beta(B_\eps(\zeta,C_0r_{\zeta,k}))}{C_0r_{\zeta,k}}. 
    \]
    Since $\{B_\eps(\zeta,C_0r_{\zeta,k})\}_{\zeta\in\partial^*_k E}$ covers $\partial^*_k E$, it follows from the 5-covering lemma that there exists a pairwise disjoint subcollection $\{B_i\}_{i\in I\subset\N}$ such that $\partial^*_k E\subset\bigcup_{i\in I} 5B_i$.  Since $r_{\zeta,k}<1/(C_0j_k)$, we have that $B_i\subset X_{j_k}$ for each $i$.  Then, by \eqref{eq:EtaSmall}, the doubling property of $\mu_\beta$, and by $r_{\zeta,k}<\delta/(5C_0)$, it follows that  
    \begin{align*}
        2^{-2k}\eta>\int_{X_{j_k}}g_u\,d\mu_\beta\ge\sum_{i\in I}\int_{B_i}g_u\,d\mu_\beta\gtrsim\frac{1}{2^k}\sum_{i\in I}\frac{\mu_\beta(B_i)}{\rad(B_i)}&\gtrsim\frac{1}{2^k}\sum_{i\in I}\frac{\mu_\beta(5B_i)}{\rad(5B_i)}\ge\frac{1}{2^k}\Ha^{-1}_{\mu_\beta,\,\delta}(\partial^*_k E).
    \end{align*}
    Hence, from \eqref{eq:Meas Theoretic Boundary Layers} we have that
    \[
    \Ha^{-1}_{\mu_\beta,\,\delta}(\partial^*E)\le\sum_{k=N_\delta}^\infty\Ha^{-1}_{\mu_\beta,\,\delta}(\partial^*_k E)\lesssim\eta.
    \]
    As $\delta>0$ and $\eta>0$ are arbitrary, we have from Lemma~\ref{lem:CodimComp} that 
    \[
    \Ha^{-s}_\nu(\partial^*E)=\Ha^{-1}_{\mu_\beta}(\partial^*E)=0.\qedhere
    \]
\end{proof}

Corollary~\ref{cor:Visintin Improvement} now follows immediately from Theorem~\ref{thm:MinkowskiSufficient} and Theorem~\ref{thm:CodimHausZero}.  As mentioned above, the first inequality in Corollary~\ref{cor:Visintin Improvement} was obtained in the Euclidean setting in \cite{V}.  While \eqref{eq:Codim Ineq} may hold with equality in the above corollary, as is the case with the von Koch snowflake domain, see \cite{Lom}, we provide an example in the next section showing that these inequalities may be strict.

\section{Examples}\label{sec:codim Examples}
In this section, we demonstrate some sharpness of the above results, and show by example that the converses of Theorem~\ref{thm:MinkowskiSufficient} and Theorem~\ref{thm:CodimHausZero} do not hold in general.  We do this by constructing a fat Cantor set as follows:

Let $I=[0,1]\subset\R$ be equipped with the Euclidean distance and Lebesgue measure.  For a parameter $0<a<1/3$, we remove an open interval $I_{1,1}$ of length $a$ from the center of $I$.  Let $C_{1,1}$ and $C_{1,2}$ be the disjoint closed intervals which comprise $I\setminus I_{1,1}$.  For the second stage of the construction, we remove an open interval $I_{2,k}$ of length $a^2$ from the center of each $C_{1,k}$, $k=1,2$.  Let $\{C_{2,k}\}_{k=1}^4$ be the remaining closed intervals.  Continuing inductively in this manner, for the $j$-th stage of the construction, we remove an open interval $I_{j,k}$ of length $a^j$ from the center of each $C_{j-1,k}$, $k=1,\dots,2^{j-1}$, and we let $\{C_{j,k}\}_{k=1}^{2^j}$ be the collection of remaining closed intervals.  Then for each $1\le k\le 2^j$, we have that 
\begin{equation}\label{eq:length C_{j,k}}
    \Ha^1(C_{j,k})=\frac{1-3a+a(2a)^j}{2^j(1-2a)}=:c_j.
\end{equation}
We then define the Cantor set $C_a$ by 
\[
C_a:=\bigcap_{j=1}^\infty\bigcup_{k=1}^{2^j}C_{j,k}=I\setminus\left(\bigcup_{j=1}^\infty\bigcup_{k=1}^{2^{j-1}}I_{j,k}\right).
\]
Since $0<a<1/3$, we see that 
\[
\Ha^1(C_a)=\frac{1-3a}{1-2a}>0.
\]

We now relate the Besov energy of the characteristic function of $C_a$ to the parameter $a$.
\begin{lem}\label{lem:Besov Cantor upper bound}
    Let $0<s<1$, $0<a<1/3$, and let $C_a$ be constructed as above.  Then
    \[
    \|\chi_{C_a}\|_{B^s_{1,1}(I)}\lesssim \sum_{j=1}^\infty (2a^{1-s})^j
    \]
    with the comparison constant depending only on $s$.
\end{lem}

\begin{proof}
    We have that 
    \[
    \|\chi_{C_a}\|_{B^s_{1,1}(I)}=\|\chi_{I\setminus C_a}\|_{B^s_{1,1}(I)}=\Big\|\sum_{j=1}^\infty\sum_{k=1}^{2^{j-1}}\chi_{I_{j,k}}\Big\|_{B^s_{1,1}(I)}\le\sum_{j=1}^\infty\sum_{k=1}^{2^{j-1}}\|\chi_{I_{j,k}}\|_{B^s_{1,1}(I)},
    \]
    and for $1\le j<\infty$ and $1\le k\le 2^{j-1}$, 
    \[
    \|\chi_{I_{j,k}}\|_{B^s_{1,1}(I)}=2\int_{I_{j,k}}\int_{I\setminus I_{j,k}}\frac{1}{|y-x|^{1+s}}dydx.
    \]
    We have that 
    \begin{align*}
    \int_{I_{j,k}}\int_{I\setminus I_{j,k}}\frac{1}{|y-x|^{1+s}}dydx&\le2\int_0^{a^j}\int_{a^j}^1\frac{1}{|y-x|^{1+s}}dydx\\
        &=\frac{2}{s(1-s)}(a^{j(1-s)}+(1-a^j)^{1-s}-1)\le\frac{2a^{j(1-s)}}{s(1-s)},
    \end{align*}
    and so it follows that 
    \[
    \|\chi_{C_a}\|_{B^s_{1,1}(I)}\le\frac{2}{s(1-s)}\sum_{j=1}^\infty(2a^{1-s})^j.\qedhere
    \]
\end{proof}

\begin{prop}\label{thm:Cantor Besov energy}
Let $0<s<1$, and let $0<a<1/3$ be such that $a<1/2^{1+s}$.  Then
\[
\|\chi_{C_a}\|_{B^s_{1,1}(I)}\simeq\sum_{j=1}^\infty (2a^{1-s})^j,
\]
with comparison constant depending only on $s$ and $a$.
\end{prop}

\begin{proof}
    By Lemma~\ref{lem:Besov Cantor upper bound}, we have that 
    \[
    \|\chi_{C_a}\|_{B^s_{1,1}(I)}\lesssim\sum_{j=1}^\infty(2a^{1-s})^j.
    \]
    To prove the converse inequality, we have that 
    \begin{equation}\label{eq:Besov Cantor Lower Bound 1}
    \|\chi_{C_a}\|_{B^s_{1,1}(I)}=2\sum_{j=1}^\infty\sum_{k=1}^{2^{j-1}}\int_{I_{j,k}}\int_{C_a}\frac{1}{|y-x|^{1+s}}dydx,
    \end{equation}
    and for each $1\le j<\infty$ and $1\le k\le 2^{j-1}$,
    \begin{align*}
        \int_{I_{j,k}}\int_{C_a}\frac{1}{|y-x|^{1+s}}dydx=\int_{I_{j,k}}\int_{I\setminus I_{j,k}}\frac{1}{|y-x|^{1+s}}dydx-\sum_{l=1}^\infty\sum_{\substack{m=1,\\ (l,m)\ne(j,k)}}^{2^{l-1}}\int_{I_{j,k}}\int_{I_{l,m}}\frac{1}{|y-x|^{1+s}}dydx.
    \end{align*}
    Fix $1\le j<\infty$ and $1\le k\le 2^{j-1}$.  Setting 
    $    J_1:=\int_{I_{j,k}}\int_{I\setminus I_{j,k}}\frac{1}{|y-x|^{1+s}}dydx$, 
    and 
    \begin{align*}
    \sum_{l=1}^\infty\sum_{\substack{m=1,\\ (l,m)\ne(j,k)}}^{2^{l-1}}&\int_{I_{j,k}}\int_{I_{l,m}}\frac{1}{|y-x|^{1+s}}dydx\\
        &=\sum_{l=1}^j\sum_{\substack{m=1,\\ (l,m)\ne(j,k)}}^{2^{l-1}}\int_{I_{j,k}}\int_{I_{l,m}}\frac{1}{|y-x|^{1+s}}dydx+\sum_{l=j+1}^\infty\sum_{m=1}^{2^{l-1}}\int_{I_{j,k}}\int_{I_{l,m}}\frac{1}{|y-x|^{1+s}}dydx\\
        &=:J_2+J_3,
    \end{align*}
    we then have that 
    \begin{equation}\label{eq:J1 J2 J3}
        \int_{I_{j,k}}\int_{C_a}\frac{1}{|y-x|^{1+s}}dydx=J_1-J_2-J_3.
    \end{equation}

    For each $1\le l\le j$ and $1\le m\le 2^{l-1}$ such that $(l,m)\ne (j,k)$, it follows from the construction that there is at least one closed interval from the collection $\{C_{j,k}\}_{k=1}^{2^j}$ separating $I_{j,k}$ and $I_{l,m}$.  Hence 
    \[
    \dist (I_{j,k},I_{l,m})\ge c_j>a^j,
    \]
    since $0<a<1/3$, where $c_j$ is given by \eqref{eq:length C_{j,k}}.  Letting $I_{j,k}\pm a^j$ denote the $a^j$-neighborhood of the interval $I_{j,k}$, it then follows that 
    \begin{align}\label{eq:J_1-J_2}
        J_1-J_2&=\int_{I_{j,k}}\int_{I\setminus I_{j,k}}\frac{1}{|y-x|^{1+s}}dydx-\sum_{l=1}^j\sum_{\substack{m=1,\\ (l,m)\ne(j,k)}}^{2^{l-1}}\int_{I_{j,k}}\int_{I_{l,m}}\frac{1}{|y-x|^{1+s}}dydx\nonumber\\
        &\ge\int_{I_{j,k}}\int_{(I_{j,k}\pm a^j)\setminus I_{j,k}}\frac{1}{|y-x|^{1+s}}dydx\nonumber\\
        &=2\int_0^{a^j}\int_{a^j}^{2a^j}\frac{1}{|y-x|^{1+s}}dydx=\frac{4(1-2^{-s})}{s(1-s)} a^{j(1-s)}.
    \end{align}

    To estimate $J_3$, fix $l\ge j+1$ and consider the collection of intervals $\{I_{l,m}\}_{m=1}^{2^{l-1}}$. By the construction of $C_a$, it follows that for each interval $I_{l,m}$ in this collection, there exist some positive integer $1\le n_m\le 2^l$ such that there are precisely $n_m$ closed intervals from the collection $\{C_{l,i}\}_{i=1}^{2^l}$ between $I_{j,k}$ and $I_{l,m}$, and so 
    \[
    \dist (I_{j,k},I_{l,m})\ge n_mc_l.
    \]
    From this, we see that 
    \[
    \int_{I_{j,k}}\int_{I_{l,m}}\frac{1}{|y-x|^{1+s}}dydx\le\frac{a^ja^l}{(n_mc_l)^{1+s}}.
    \]
    Furthermore, for each $1\le n\le 2^l$, there are at most two intervals from the collection $\{I_{l,m}\}_{m=1}^{2^{l-1}}$ which are separated from $I_{j,k}$ by precisely $n$ closed intervals from $\{C_{l,i}\}_{i=1}^{2^l}$.  Hence, we overestimate $J_3$ as follows:
    \begin{align*}
        J_3&=\sum_{l=j+1}^\infty\sum_{m=1}^{2^{l-1}}\int_{I_{j,k}}\int_{I_{l,m}}\frac{1}{|y-x|^{1+s}}dydx\le 2\sum_{l=j+1}^\infty\sum_{n=1}^{2^l}\frac{a^ja^l}{(nc_l)^{1+s}}=2a^j\sum_{l=j+1}^\infty\frac{a^l}{c_l^{1+s}}\sum_{n=1}^{2^l}\frac{1}{n^{1+s}}
    \end{align*}
We have that 
\begin{align*}
    \frac{a^l}{c_l^{1+s}}=a^l\left(\frac{2^l(1-2a)}{1-3a+a(2a)^l}\right)^{1+s}\le(2^{1+s}a)^l\frac{1}{(1-3a)^{1+s}},
\end{align*}
and so by the assumption that $a<\min\{1/3, 1/2^{1+s}\}$, it follows that 
\[
C_{s,a}:=\sum_{l=1}^\infty \frac{a_l}{c_l^{1+s}}\sum_{n=1}^\infty\frac{1}{n^{1+s}}<\infty.
\]
Therefore, we have that 
\begin{align}\label{eq:J3}
    J_3\le 2C_{s,a}a^{j}.
\end{align}

Combining \eqref{eq:J1 J2 J3}, \eqref{eq:J_1-J_2}, and \eqref{eq:J3}, we have that 
\[
\int_{I_{j,k}}\int_{C_a}\frac{1}{|y-x|^{1+s}}dydx\ge\left(\frac{4(1-2^{-s})}{s(1-s)}-2C_{s,a}a^{js}\right)a^{j(1-s)},
\]
and so there exists $N\in\N$, depending on $s$ and $a$ such that 
\[
\int_{I_{j,k}}\int_{C_a}\frac{1}{|y-x|^{1+s}}dydx\ge\frac{2(1-s^{-s})}{s(1-s)}a^{j(1-s)}=:C_sa^{j(1-s)}
\]
for all $j>N$.  By \eqref{eq:Besov Cantor Lower Bound 1}, we then have that 
\begin{align*}
    \|\chi_{C_a}\|_{B^s_{1,1}(I)}\ge 2\sum_{j=1}^N\sum_{k=1}^{2^{j-1}}\int_{I_{j,k}}\int_{C_a}\frac{1}{|y-x|^{1+s}}dydx+C_s\sum_{j=N+1}^\infty(2a^{1-s})^j.
\end{align*}
For each $1\le j\le N$ and $1\le k\le 2^{j-1}$, there exists $C_{j,k}>0$ such that 
\[
2\int_{I_{j,k}}\int_{C_a}\frac{1}{|y-x|^{1+s}}dydx\ge C_{j,k}a^{j(1-s)}.
\]
Such a $C_{j,k}$ exists since $\Ha^1(C_a)>0$.  Therefore, letting 
\[
C:=\min\{C_s,C_{j,k}:1\le j\le N,\,1\le k\le 2^{j-1}\},
\]
which depends only on $s$ and $a$, we have that
\[
\|\chi_{C_a}\|_{B^s_{1,1}(I)}\ge C\sum_{j=1}^\infty(2a^{1-s})^j.\qedhere
\] 
\end{proof}

By using Proposition~\ref{thm:Cantor Besov energy}, we can now provide examples illustrating the sharpness of Theorem~\ref{thm:CodimHausZero}, and that the converses of Theorem~\ref{thm:CodimHausZero} and Theorem~\ref{thm:MinkowskiSufficient} do not hold in general.

\begin{example}\label{ex:Haus Sharp}
    For each $0<a<1/3$, consider the Cantor set $C_a$ constructed above.   The regularized boundary of $C_a$ is given by 
    \[    
    \partial^+C_a=C_a.
    \] 
    Then,
    \[
    \Ha^1(\partial^+ C_a)=\Ha^1(C_a)=\frac{1-3a}{1-2a}>0,
    \]
    and so $\Ha^{-s}(\partial^+C_a)=\infty$ for all $0<s<1$.  However, by Proposition~\ref{thm:Cantor Besov energy}, we can choose $0<a<1/3$ and $0<s<1$ such that $\chi_{C_a}\in B^s_{1,1}(I)$.  This shows that Theorem~\ref{thm:CodimHausZero} is sharp in the sense that we cannot strengthen the conclusion by replacing $\partial^*E$ with $\partial^+E$, or even replacing $\Ha^{-s}(\partial^*E)=0$ with $\Ha^{-s}(\partial^+E)<\infty$.
\end{example}

\begin{example}\label{ex:Haus Converse}
For each $0<a<1/3$, consider the Cantor set $C_a$, and let $\{x_i\}_{i\in\N}$ be the endpoints of the open intervals $I_{j,k}$ considered above.  Since $\Ha^1(C_a)>0$, we see from the construction that 
\[
\limsup_{r\to 0^+}\frac{\Ha^1(B(x,r)\setminus C_a)}{\Ha^1(B(x,r))}=0
\]
for all $x\in C_a\setminus\{x_i\}_{i\in\N}.$  Therefore, $\partial^*C_a=\{x_i\}_{i\in\N}$, and as a countable set, it follows that 
\[
\Ha^{-s}(\partial^*C_a)=0
\]
for all $0<s<1$.  

Now, fix $0<a_0<1/4$.  We can then choose $0<s_0<1$ sufficiently close to $1$ so that $2a_0^{1-s_0}>1$.  Since $a_0^{1/(1+s_0)}<a_0^{1/2}<1/2$, the hypotheses of Proposition~\ref{thm:Cantor Besov energy} are satisfied, and so $\chi_{C_{a_0}}\not\in B^{s_0}_{1,1}(I)$.  However, by the above discussion, we have that $\Ha^{-s_0}(\partial^* C_{a_0})=0$. Thus the converse of Theorem~\ref{thm:CodimHausZero} does not hold in general.   
\end{example}

\begin{example}\label{ex:Mink Converse}
    Let $0<a<1/3$ and $0<t<1$.  For $r>0$, consider a countable cover $\{B(y_j,r)\}_{j\in\N}$ of $\partial^*C_a=\{x_i\}_{i\in\N}$, such that $y_j\in\partial^*C_a$ for each $j\in\N$.  If $z_1\in C_a\setminus \bigcup_j B(y_j,r)$, then by the density of $\partial^*C_a$ in $C_a$, we have that $\dist(z_1,\bigcup_jB(y_j,r))=0$.  Hence if $z_2\in C_a\setminus \bigcup_jB(y_j,r)$ such that $z_1\ne z_2$, then $d(z_1,z_2)\ge 2r$.  Therefore $C_a\setminus\bigcup_jB(y_j,r)$ is a finite set.  Thus, we have that 
    \[
    \sum_j r^{-t}\Ha^1(B(y_j,r))\ge r^{-t}\Ha^1(C_a),
    \]
    and so since $\{B(y_j,r)\}_{j\in\N}$ is arbitrary, we have that $\M^{-t}_r(\partial^*C_a)\ge r^{-t}\Ha^1(C_a)$.  Since $\Ha^1(C_a)>0$, it follows that $\underline\M^{-t}(\partial^*C_a)=\infty$.  However by Proposition~\ref{thm:Cantor Besov energy}, we can choose $0<a<1/3$ and $0<s<1/2$ so that $\chi_{C_a}\in B^s_{1,1}(I)$.  Therefore, the converse to Theorem~\ref{thm:MinkowskiSufficient} does not hold in general for any $0<t<1$, nor even with the weaker conclusion that $\underline\M^{-t}(\partial^*E)<\infty$. 
\end{example}

\begin{example}\label{ex:Codim Strict}
    It was shown in \cite{Lom} that the von Koch snowflake domain $K\subset\R^2$ is such that 
    \[    \overline\codim_\M(\partial^+K)=\codim_F(K)=\codim_\Ha(\partial^*K)=2-\log4/\log3.
    \]
    That is, \eqref{eq:Codim Ineq} holds with equality in some cases.  However, the fat Cantor sets constructed above show that these inequalities may be strict.  For example, let $a=1/4$ and consider $C_a$. By Proposition~\ref{thm:Cantor Besov energy}, we see that $\chi_{C_a}\in B^s_{1,1}(I)$ if and only if $0<s<1/2$, and so $\codim_F(C_a)=1/2$.  However, $\Ha^{-s}(\partial^*C_a)=0$ and $\overline\M^{-s}(\partial^+C_a)=\infty$ for all $0<s<1$.  Hence, we have that 
    \[
    0=\overline\codim_\M(\partial^+C_a)<\codim_F(C_a)<\codim_\Ha(\partial^*C_a)=1.
    \]
\end{example}

\section{Existence of minimizers of $\J_\Omega^s$}\label{sec:Existence of Minimizers}
Throughout this section, we assume that $(X,d,\mu)$ is a  complete metric measure space, with $\mu$ a doubling Borel regular measure.  Let $\Omega\subset X$ be a bounded domain, that is, a bounded connected open set.  Without loss of generality, we may assume that $X\setminus\Omega\ne\varnothing$.  Otherwise, the existence results of this section hold trivially.  Let $F\subset X$ be a measurable set such that $F\setminus\Omega\ne\varnothing$, and let $0<s<1$.  We consider the functional 
\[
\J_\Omega^s(E):=L_s(E\cap\Omega,X\setminus E)+L_s(E\setminus\Omega,\Omega\setminus E)
\]
and the associated minimization problem defined in Section~\ref{sec:Fractional Perimeter} and Definition~\ref{def:Fractional Minimizer}, and prove existence of minimizers of $\J_\Omega^s$ by the direct method of calculus of variations.  We first note that $\J_\Omega^s$ is lower semicontinuous with respect to $L^1$-convergence by Fatou's lemma, see \cite[Proposition~3.1]{CRS}.  

\begin{prop}\label{prop:Nonlocal Lower Semicontinuity}
    If $E_k$ and $E$ are subsets of $X$ such that $\chi_{E_k}\to\chi_E$ in $L^1_\loc(X)$ as $k\to\infty$, then 
    \[    \J_\Omega^s(E)\le\liminf_{k\to\infty}\J_\Omega^s(E_k).
    \]
\end{prop}

The following compactness theorem is an adaption of \cite[Theorem~7.1]{Hitch}, proved there in $\R^n$.  To prove the result in the doubling metric measure space setting, we will use the discrete convolution technique, obtained from the Lipschitz partition of unity given by Lemmas~\ref{lem:BoundedOverlapCover} and \ref{lem:Lipschitz POU}.   

\begin{thm}\label{thm:BesovCompactness}
    Let $\Omega\subset X$ be an open set.  Let $\F$ be a family of functions on $\Omega$ such that for all compact sets $K\subset\Omega$,
    \begin{equation}\label{eq:L1 bound}    \sup_{f\in\F}\|f\|_{L^1(K)}<\infty,
    \end{equation}
    and also that 
    \begin{equation}\label{eq:BesovEnergyBound}
    \sup_{f\in\F}\|f\|_{B^s_{1,1}(K)}<\infty.
    \end{equation}
    Then, for any sequence $\{f_k\}_k\subset\F$, there exists $f\in L^1_\loc(\Omega)$ and a subsequence, not relabeled, so that $f_k\to f$ in $L^1_\loc(\Omega)$ and pointwise $\mu$-a.e.\ in $\Omega$.
    % Then $\F$ is precompact in $L^1(\wtil\Omega)$ for all open sets $\wtil\Omega\Subset\Omega$.
\end{thm}

\begin{proof}
    Let $\widetilde\Omega\Subset\Omega$, and let $K\subset\Omega$ be compact such that $\widetilde\Omega\subset K$. For each $0<\eps<\dist(K,X\setminus\Omega)/4$, there exists by Lemma~\ref{lem:BoundedOverlapCover}, a countable cover $\{B_i^\eps\}_{i\in I\subset\N}$ of $X$ by 
    balls of radius $\eps$ with bounded overlap, such that collection $\{\frac{1}{5}B_i^\eps\}_{i\in I}$ is pairwise disjoint.  Let $\{\phii_i^\eps\}_{i\in I}$ be a Lipschitz partition of unity subordinate to this cover, given by Lemma~\ref{lem:Lipschitz POU}. 
    
    Let $J:=\{i\in I: B_i^\eps\cap K\ne\varnothing\}$.  Since $K$ is bounded and since $\{\frac{1}{5}B_i^\eps\}_{i\in J}$ is pairwise disjoint, the doubling property of $\mu$ ensures that $J$ is a finite set. Thus, there exists $N_\eps\in\N$ so that, relabeling if necessary, we can denumerate this collection as follows:
    \[
    \{B_i^\eps\}_{i\in J}=\{B_i^\eps\}_{i=1}^{N_\eps}.
    \]
    By Lemma~\ref{lem:Lipschitz POU}, it follows that for each $x\in K$, 
    \begin{equation}\label{eq:POU on K}
        \sum_{i=1}^{N_\eps}\phii_i^\eps(x)=\sum_{i\in I}\phii_i^\eps(x)=1.
    \end{equation}
    By the above upper bound on $\eps$, we can choose a compact set $\wtil K\Subset\Omega$, independent of $\eps$, such that $\bigcup_{i=1}^{N_\eps}2B_i^\eps\subset\wtil K$.  For each $f\in\F$, we define the discrete convolution $f_\eps$ of $f$ by 
    \[    f_\eps=\sum_{i=1}^{N_\eps}f_{B_i^\eps}\phii_i^\eps.
    \]
    We also define the vector
    \[
    R(f)=\left(\int_{B_1^\eps}f\,d\mu,\dots,\int_{B_{N^\eps}^\eps}f\,d\mu\right)\in\R^{N_\eps}.
    \]
    By the bounded overlap of the cover $\{B_{i}^\eps\}_{i=1}^{N_\eps}$ and \eqref{eq:L1 bound}, we have that 
    \begin{align*}
        \|R(f)\|_{\ell^1}\le\sum_{i=1}^{N_\eps}\int_{B_{i}^\eps}|f|d\mu\lesssim\int_{\bigcup_{i=1}^{N_\eps}B^\eps_i}|f|d\mu\le\int_{\wtil K}|f|d\mu\le\sup_{f\in\F}\|f\|_{L^1(\wtil K)}<\infty.
    \end{align*}
    Therefore, we have that 
    $R(\F)$ is a totally bounded set in $(\R^{N_\eps},\|\cdot\|_{\ell^1})$, and so for every $\eta>0$, there exists $M_\eta\in\N$ and $b_1,\dots,b_{M_\eta}\in\R^{N_\eps}$ such that for every $f\in\F$, there exists $j\in\{1,\dots,M_\eta\}$ such that 
    \begin{equation}\label{eq:b_j}
    \|R(f)-b_j\|_1<\eta.
    \end{equation}

    Let $\eta:=\eps/(N_\eps+1)$.
    For each $j\in\{1,\dots,M_\eta\}$, let $b_j=(b_{1,j},\dots,b_{N_\eps,j})$, and define $\beta_j:\Omega\to\R$ by
    \[
    \beta_j=\sum_{i=1}^{N_\eps}\frac{b_{i,j}}{\mu(B_i^\eps)}\phii_i^\eps.
    \]
    For $f\in\F$, let $j\in\{1,\dots,M_\eta\}$ be such that \eqref{eq:b_j} holds.  We then have that
    \begin{equation}\label{eq:L^1(K) TI}
    \|f-\beta_j\|_{L^1(K)}\le\|f-f_\eps\|_{L^1(K)}+\|f_\eps-\beta_j\|_{L^1(K)}.
    \end{equation}
    By \eqref{eq:POU on K}, we have that 
    \begin{align*}
        \|f-f_\eps\|_{L^1(K)}\le\sum_{k=1}^{N_\eps}\int_{B_k^\eps\cap K}|f-f_\eps|d\mu\le\sum_{k=1}^{N_\eps}\sum_{i=1}^{N_\eps}\int_{B_k^\eps\cap K}|f(x)-f_{B_i^\eps}|\phii_i^\eps(x)d\mu(x).
    \end{align*}
    Letting $I_k:=\{i\in\{1,\dots, N_\eps\}:2B_i^\eps\cap B_k^\eps\ne\varnothing\}$, we have that 
    \[
    \|f-f_\eps\|_{L^1(K)}\le\sum_{k=1}^{N_\eps}\sum_{i\in I_k}\int_{B_k^\eps\cap K}|f-f_{B_i^\eps}|d\mu\le\sum_{k=1}^{N_\eps}\sum_{i\in I_k}\int_{B_k^\eps}\fint_{B_i^\eps}|f(x)-f(y)|d\mu(y)d\mu(x).
    \]
    For $x\in B_k^\eps$ and $i\in I_k$, we then have by the doubling property of $\mu$ that 
    \begin{align*}
        \fint_{B_i^\eps}|f(x)-f(y)|d\mu(y)&\lesssim\eps^s\frac{\mu(B(x,5\eps))}{\mu(B_i^\eps)}\int_{B_i^\eps}\frac{|f(y)-f(x)|}{d(x,y)^s\mu(B(x,d(x,y)))}d\mu(y)\\
        &\lesssim\eps^s\int_{B_i^\eps}\frac{|f(y)-f(x)|}{d(x,y)^s\mu(B(x,d(x,y)))}d\mu(y).
    \end{align*}
    Substituting this into the previous expression, we have by bounded overlap of $\{B_i^\eps\}_{i=1}^{N_\eps}$, 
    \begin{align}\label{eq:L^1(K) first piece}
        \|f-f_\eps\|_{L^1(K)}&\lesssim\eps^s\sum_{k=1}^{N_\eps}\sum_{i\in I_k}\int_{B_k^\eps}\int_{B_i^\eps}\frac{|f(y)-f(x)|}{d(x,y)^s\mu(B(x,d(x,y)))}d\mu(y)d\mu(x)\nonumber\\
        &\lesssim\eps^s\int_{\wtil K}\int_{\wtil K}\frac{|f(y)-f(x)|}{d(x,y)^s\mu(B(x,d(x,y)))}d\mu(y)d\mu(x)\nonumber\\
        &\le\eps^s\sup_{f\in\F}\|f\|_{B^s_{1,1}(\wtil K)}.
    \end{align}

    Since we have chosen $j\in\{1,\dots,M_\eta\}$ so that \eqref{eq:b_j} holds, we have from the doubling property of $\mu$ and choice of $\eta$ that
    \begin{align*}
        \|f_\eps-\beta_j\|_{L^1(K)}\le\sum_{k=1}^{N_\eps}\int_{B_k^\eps}|f_\eps-\beta_j|d\mu&\le\sum_{k=1}^{N_\eps}\int_{B_k^\eps}\sum_{i=1}^{N_\eps}|f_{B_i^\eps}-b_{i,j}/\mu(B_i^\eps)|\phii_i^\eps(x)d\mu(x)\\
        &\le\sum_{k=1}^{N_\eps}\int_{B_k^\eps}\sum_{i\in I_k}\frac{1}{\mu(B_i^\eps)}\left|\int_{B_i^\eps}f\,d\mu-b_{i,j}\right|d\mu(x)\\
        &\lesssim\sum_{k=1}^{N_\eps}\fint_{B_k^\eps}\sum_{i=1}^{N_\eps}\left|\int_{B_i^\eps}f\,d\mu-b_{i,j}\right|d\mu(x)\\
        &=\sum_{k=1}^{N_\eps}\fint_{B_k^\eps}\|R(f)-b_j\|_{\ell^1}\,d\mu<N_\eps\eta<\eps.
    \end{align*}
    Combining this estimate with \eqref{eq:L^1(K) first piece} and \eqref{eq:L^1(K) TI}, we have that 
    \[
    \|f-\beta_j\|_{L^1(\wtil\Omega)}\le\|f-\beta_j\|_{L^1(K)}\lesssim\eps^s\sup_{f\in\F}\|f\|_{B^s_{1,1}(\wtil K)}+\eps.
    \]
Recall that the compact set $\wtil K$ is chosen independently of $\eps>0$, and \[
C:=\sup_{f\in\F}\|f\|_{B^s_{1,1}(\wtil K)}<\infty
\]
by \eqref{eq:BesovEnergyBound}. 
 Therefore, we have shown that for all $\eps>0$, there exists $M_\eps\in\N$ and $\beta_1,\dots,\beta_{M_\eps}\in L^1(\wtil\Omega)$ such that for all $f\in\F$, there exists $j\in \{1,\dots, M_\eps\}$ such that 
 \[
 \|f-\beta_j\|_{L^1(\wtil\Omega)}\lesssim C\eps^s+\eps,
 \]
 with comparison constant depending only on $s$ and the doubling constant of $\mu$.
 That is, $\F$ is totally bounded, and hence precompact, in $L^1(\wtil\Omega)$, by completeness of $L^1(\wtil\Omega)$.

 Let $\{f_k\}_k\subset\F$.  Let $\Omega=\bigcup_{i=1}^\infty\Omega_i$ be an exhaustion of $\Omega$ by an increasing sequence of open sets $\Omega_i\Subset\Omega$.  Since $\F$ is precompact in $L^1(\Omega_1)$, there exists a subsequence $\{f_{1,k}\}_k\subset\{f_k\}_k$ and $g_1$ such that $f_{1,k}\to g_1$ in $L^1(\Omega_1)$ as $k\to\infty$.  We choose $k_1\in\N$ such that 
 \[
 \|f_{1,k_1}-g_1\|_{L^1(\Omega_1)}<2^{-1}.
 \]
 Inductively, for $i\ge 2$, there exists by precompactness of $\F$ on $\Omega_i$, a subsequence $\{f_{i,k}\}_k\subset\{f_{i-1,k}\}_{k=k_{i-1}}^\infty$ and a function $g_i$ such that $f_{i,k}\to g_i$ in $L^1(\Omega_i)$ as $k\to\infty$.  Furthermore, there exists $k_i\in\N$ such that
 \[
 \|f_{i,k_i}-g_i\|_{L^1(\Omega_i)}<2^{-i}.
 \]
 Since the subsequences are nested in this manner, we have that $g_i=g_{i-1}$ $\mu$-a.e.\ on $\Omega_{i-1}$.  We then define $f\in L^1_\loc(\Omega)$ by $f(x)=g_i(x)$ if $x\in\Omega_i$. We claim that $f_{i,k_i}\to f$ in $L^1_\loc(\Omega)$ and pointwise $\mu$-a.e.\ in $\Omega$.  Indeed, for $i_0\in\N$, $\{f_{i,k_i}\}_{i=i_0}^\infty$ is a subsequence of $\{f_{i_0,k}\}_k$, and  $f_{i_0,k}\to f$ in $L^1(\Omega_{i_0})$ as $k\to\infty$.  Moreover, by the monotone convergence theorem and our choice of $k_i$, we have that 
 \[
 \int_{\Omega_{i_0}}\sum_{i=i_0}^\infty|f_{i,k_i}-f|d\mu=\sum_{i=i_0}^\infty\int_{\Omega_{i_0}}|f_{i,k_i}-f|d\mu<\sum_{i=i_0}^\infty2^{-i}<\infty.
 \]
 Hence, $\sum_{i=i_0}^\infty|f_{i,k_i}-f|<\infty$ $\mu$-a.e.\ in $\Omega_{i_0}$, and so $f_{i,k_i}\to f$ pointwise $\mu$-a.e.\ in $\Omega_{i_0}$ as $i\to\infty$.  Since $i_0$ is arbitrary, and $\Omega=\bigcup_{i=1}^\infty\Omega_i$ with $\Omega_i\Subset\Omega$, this concludes the proof.
\end{proof}

We define the following collection of measurable sets:
\[
\A(F):=\{E\subset X:E\setminus\Omega=F\setminus\Omega,\,\J_\Omega^s(E)<\infty\}.
\]

\begin{thm}\label{thm:Existence of s-minimizers}
If $\A(F)\ne\varnothing$, then there exists a minimizer of $\J_\Omega^s$.
\end{thm}

\begin{proof}
    Since $\A(F)\ne\varnothing$, we have that $\beta:=\inf\{\J_\Omega^s(E):E\in\A(F)\}<\infty$.  Let $\{E_k\}_{k\in\N}$ be a sequence in $\A(F)$ such that $\J_\Omega^s(E_k)\to\beta$ as $k\to\infty$.  Hence we have that 
    \[
    \sup_k\left(\|\chi_{E_k}\|_{L^1(\Omega)}+\|\chi_{E_k}\|_{B^s_{1,1}(\Omega)}\right)<\infty
    \]
    as $\Omega$ is bounded,
    and so by Theorem~\ref{thm:BesovCompactness}, there exists $f\in L^1_\loc(\Omega)$ and a subsequence $\{E_k\}_k$, not relabeled, such that $\chi_{E_k}\to f$
    in $L^1_\loc(\Omega)$ and pointwise $\mu$-a.e.\ in $\Omega$ as $k\to\infty$.  Thus, we can take $f=\chi_{\wtil E}$ for some set $\wtil E\subset\Omega$.  
    
    Let $\Omega=\bigcup_{i=1}^\infty\Omega_i$ be an exhaustion of $\Omega$ by an increasing sequence of open sets $\Omega_i\Subset\Omega$, and let $\eps>0$. Since $\Omega$ is bounded, there exists $i\in\N$ such that $\mu(\Omega\setminus\Omega_i)<\eps$.  We then have that
    \[
    \int_\Omega|\chi_{E_k}-\chi_{\wtil E}|d\mu=\int_{\Omega_i}|\chi_{E_k}-f|d\mu+\int_{\Omega\setminus\Omega_i}|\chi_{E_k}-\chi_{\wtil E}|d\mu<\int_{\Omega_i}|\chi_{E_k}-f|d\mu+\eps.
    \]
    Since $\chi_{E_k}\to f$ in $L^1_\loc(\Omega)$, we have that $\chi_{E_k}\to\chi_{\wtil E}$ in $L^1(\Omega)$. Letting $E=\wtil E\cup(F\setminus\Omega)$, we then have that $\chi_{E_k}\to\chi_E$ in $L^1_\loc(X)$, as $\chi_E=\chi_{E_k}=\chi_F$ in $X\setminus\Omega$ for all $k$. By Proposition~\ref{prop:Nonlocal Lower Semicontinuity}, it follows that 
    \[    \J_\Omega^s(E)\le\liminf_{k\to\infty}\J_\Omega^s(E_k)=\beta<\infty.
    \]
    As $E\in\A(F)$, it follows that $E$ is a minimizer of $\J_\Omega^s$.
\end{proof}

\begin{remark}
    We note that Theorem~\ref{thm:MinkowskiSufficient} gives us a sufficient condition on $\Omega$ for existence of solutions.  Namely, if $(X,d,\mu)$ is a doubling metric measure space satisfying the LLC-1 condition, and if $\Omega\subset X$ is bounded and there exists $t>s$ such that $\overline\M^{-t}(\partial^+\Omega)<\infty$, then for any measurable set $F\subset X$ with $F\setminus\Omega\ne\varnothing$, we have that $\A(F)\ne\emptyset$.  Indeed, by Theorem~\ref{thm:MinkowskiSufficient}, 
    \[
    \J_\Omega^s(F\setminus\Omega)=L_s(F\setminus\Omega,\Omega)\le L_s(X\setminus\Omega,\Omega)\simeq\|\chi_\Omega\|_{B^s_{1,1}(X)}<\infty,
    \]
    and so $F\setminus\Omega\in\A(F)$.  We compare this to the existence result from \cite{CRS} in $\R^n$, where the assumption that $\Omega\subset\R^n$ is a bounded Lipschitz domain ensures that $\|\chi_\Omega\|_{B^s_{1,1}(\R^n)}<\infty$.   
\end{remark}

\section{Regularity of minimizers of $\J_\Omega^s$}\label{sec:Regularity of Minimizers}

Throughout this section, we assume that $(X,d,\mu)$ is a complete, connected metric measure space, with $\mu$ a doubling Borel regular measure.  We let $0<s<1$ and assume that $\Omega\subset X$ is a bounded domain. As before, we may assume, without loss of generality, that $X\setminus\Omega\ne\varnothing$.  Otherwise, constant functions are the only minimizers of $\J_\Omega^s$, and the results in this section hold trivially.  We prove uniform density estimates and porosity of minimizers of $\J_\Omega^s$, analogous to the results obtained in the Euclidean setting in \cite[Theorem~4.1,Corollary~4.3]{CRS}.  In the metric setting, similar regularity results were obtained in \cite[Theorems~4.2,5.2]{KKLS} for minimizers of the (local) perimeter functional when the space is doubling and supports a $(1,1)$-Poincar\'e inequality.  We emphasize here that the energy functional $\J_\Omega^s$ is nonlocal, unlike in \cite{KKLS}. 

\subsection{Uniform density estimates}

In \cite{KKLS}, the authors obtained uniform density estimates for the (local) perimeter minimizers by means of the $(1,1)$-Poincar\'e inequality and the De Giorgi method.  Since we are concerned with the nonlocal case, we do not need to assume that $(X,d,\mu)$ supports a $(1,1)$-Poincar\'e inequality.  Instead, we will make use of the fractional Poincar\'e inequality given by Theorem~\ref{thm:FractionalPoincare}, which does not impose the geometric restrictions that the $(1,1)$-Poincar\'e inequality imposes on the metric measure space. In the statement of Theorem~\ref{thm:FractionalPoincare}, it is assumed that $(X,d,\mu)$ is reverse doubling; since we assume that $X$ is connected and $\mu$ is doubling, this condition is satisfied \cite[Corollary~3.8]{BB}.

\begin{proof}[Proof of Theorem~\ref{thm:Uniform Density}]
     
 Let $E\subset X$ be a minimizer of $\J_\Omega^s$. Modifying $E$ by a set of measure zero if necessary, we may assume by Lemma~\ref{lem:Set Representative} that $\mu(B(x,r)\cap E)>0$ and $\mu(B(x,r)\setminus E)>0$ for each $x\in\partial E$ and $r>0$.  Note that such a modification on a set of measure zero does not affect the property of being a minimizer.  Let $x_0\in\partial E\cap\Omega$ and $R_0>0$ be such that $B(x_0,2R_0)\subset\Omega$. We prove the first inequality of \eqref{eq:Uniform Density}, as the second inequality follows from the first and the fact that the complement of a minimizer of $\J_\Omega^s$ is also a minimizer. This is due to the fact that $\JOm^s(E)=\JOm^s(X\setminus E)$, by symmetry of the kernel $K_s$. 
 
 Let $z\in\Omega$ and $R>0$ be such that $B(z,4R)\subset\Omega.$ 
 For each $0<r<R$, let 
 \[
 u=\chi_{B(z,r)\cap E}.
 \]
 Since $X$ is connected and $X\setminus\Omega\ne\varnothing$, there exists $y\in\partial B(z,3r/2)$, and so $B(y,r/2)\subset B(z,2r)$ and $B(y,r/2)\cap B(z,r)=\varnothing$.  Since $\mu$ is doubling, it then follows that 
    \begin{align}\label{eq:Connectedness Application}
    \frac{\mu(B(z,2r)\cap\{|u|>0\})}{\mu(B(z,2r))}\le\frac{\mu(B(z,r))}{\mu(B(z,2r))}\le\frac{\mu(B(z,r))}{\mu(B(z,r))+\mu(B(y,r/2))}\le 1/(1+C_\mu^{-3})<1.
    \end{align}
    By the doubling property of $\mu$ and by applying Lemma~\ref{lem:FracMazya} along with \eqref{eq:LsKernel}, we obtain
    \begin{align*}
        \left(\frac{\mu(B(z,r)\cap E)}{\mu(B(z,r))}\right)^{(Q-s)/Q}&\lesssim\left(\fint_{B(z,2r)}|u|^{Q/(Q-s)}d\mu\right)^{(Q-s)/Q}\\
        &\lesssim\frac{r^s}{\mu(B(z,r))}\int_{B(z,4r)}\int_{B(z,4r)}\frac{|u(x)-u(y)|}{d(x,y)^s\mu(B(x,d(x,y))}d\mu(y)d\mu(x)\\
        &\lesssim\frac{r^s}{\mu(B(z,r))}\int_{B(z,r)\cap E}\int_{X\setminus( B(z,r)\cap E)}\frac{d\mu(y)\,d\mu(x)}{d(x,y)^s\mu(B(x,d(x,y))}\\
        &\simeq\frac{r^s}{\mu(B(z,r))} L_s(B(z,r)\cap E,X\setminus (B(z,r)\cap E)).
    \end{align*}
    Setting $A=B(z,r)\cap E$, we have by Lemma~\ref{lem:SupSubSoln} that
    \begin{align*}
    L_s(A,X\setminus A)=L_s(A,X\setminus E)+L_s(A,E\setminus A)\lesssim L_s(A,E\setminus A)\le  L_s(A,X\setminus B(z,r)).
    \end{align*}
     Substituting this into the above inequality along with \eqref{eq:LsKernel} then yields
    \begin{align*}
     \left(\frac{\mu(B(z,r)\cap E)}{\mu(B(z,r))}\right)^{(Q-s)/Q}\lesssim\frac{r^s}{\mu(B(z,r))}\int_{B(z,r)\cap E}\int_{X\setminus B(z,r)}\frac{d\mu(y)\,d\mu(x)}{d(x,y)^s\mu(B(x,d(x,y))}.
    \end{align*}
    Setting $r_x:=r-d(z,x)$ for each $x\in E\cap B(z,r)$, we have by the doubling property of $\mu$ that 
    \begin{align*}
        \int_{X\setminus B(z,r)}\frac{d\mu(y)}{d(x,y)^s\mu(B(x,d(x,y))}&\le\int_{X\setminus B(x,r_x)}\frac{d\mu(y)}{d(x,y)^s\mu(B(x,d(x,y))}\\
        &=\sum_{m=0}^\infty\int_{B(x,2^{m+1}r_x)\setminus B(x,2^mr_x)}\frac{d\mu(y)}{d(x,y)^s\mu(B(x,d(x,y))}\\
        &\lesssim\frac{1}{(r-d(z,x))^s}.
    \end{align*}
   Substituting this into the previous expression, we obtain
    \begin{align*}
        \frac{\mu(B(z,r))}{r^s}\left(\frac{\mu(B(z,r)\cap E)}{\mu(B(z,r))}\right)^{(Q-s)/Q}&\lesssim\int_{B(z,r)\cap E}\frac{1}{(r-d(z,x))^s}d\mu(x).
    \end{align*}

    Integrating both sides of the above expression with respect to $r$, and using Tonelli's theorem gives 
    \begin{align*}
        \int_{R/2}^R\frac{\mu(B(z,r))}{r^s}\left(\frac{\mu(B(z,r)\cap E)}{\mu(B(z,r))}\right)^{(Q-s)/Q}dr&\lesssim  \int_0^R\int_{B(z,r)\cap E}\frac{1}{(r-d(z,x))^s}d\mu(x)dr\\
        &=\int_0^R\int_{B(z,R)}\frac{\chi_E(x)\chi_{B(z,r)}(x)}{(r-d(z,x))^s}d\mu(x)dr\\
        &=\int_{B(z,R)}\chi_E(x)\int_0^R\frac{\chi_{B(z,r)}(x)}{(r-d(z,x))^s}drd\mu(x)\\
        &=\int_{B(z,R)}\chi_E(x)\int_{d(z,x)}^R\frac{1}{(r-d(z,x))^s}drd\mu(x)\\
        &\lesssim R^{1-s}\mu(B(z,R)\cap E).
    \end{align*}
    By the doubling property of $\mu$, we estimate the left-hand side of the above expression from below by
    \begin{align*}
       \int_{R/2}^R\frac{\mu(B(z,r))}{r^s}\left(\frac{\mu(B(z,r)\cap E)}{\mu(B(z,r))}\right)^{(Q-s)/Q}dr
       &\gtrsim R^{1-s}\mu(B(z,R))\left(\frac{\mu(B(z,R/2)\cap E)}{\mu(B(z,R/2))}\right)^{(Q-s)/Q}.
    \end{align*}
    Combining this with the previous expression, we obtain
    \begin{align}\label{eq:Iteration}
       \left(\frac{\mu(B(z,R/2)\cap E)}{\mu(B(z,R/2))}\right)^{(Q-s)/Q}&\le C_0\frac{\mu(B(z,R)\cap E)}{\mu(B(z,R))} 
    \end{align}
    for all $z\in\Omega$ and $R>0$ such that $B(z,4R)\subset\Omega$, where the constant $C_0\ge 1$ depends only on $s$ and the doubling constant of $\mu$.  

    Now consider $x_0\in\partial E\cap\Omega$ and $R_0>0$ fixed at the beginning of the proof.  Suppose that 
    \begin{equation}\label{eq:Density}
    \frac{\mu(B(x_0,R_0)\cap E)}{\mu(B(x_0,R_0)}<\gamma_0:=\frac{1}{2^{Q+1}C_0^{Q/s}C_Q}.
    \end{equation}
    Then, for $z\in B(x_0,R_0/4)$, and $j\in\N$ with $j\ge 2$, set $B_j:=B(z,R_0/2^j)$.  Setting $p:=Q/(Q-s)$, we iterate estimate \eqref{eq:Iteration} to obtain
    \begin{align*}
        \frac{\mu(B_j \cap E)}{\mu(B_j)}&\le\left(C_0\frac{\mu(B_{j-1}\cap E)}{\mu(B_{j-1})}\right)^p\\
        &\le C_0^p\left(C_0\frac{\mu(B_{j-2}\cap E)}{\mu(B_{j-2})}\right)^{p^2}\\
        &\le C_0^{p+p^2+\cdots+p^{j-1}}\left(\frac{\mu(B(z,R_0/2)\cap E)}{\mu(B(z,R_0/2))}\right)^{p^{j-1}}\\
        &\le C_0^{p+p^2+\cdots+p^{j-1}}\left(2^QC_Q\frac{\mu(B(x_0,R_0)\cap E)}{\mu(B(x_0,R_0))}\right)^{p^{j-1}}\le(2^QC_0^{Q/s}C_Q\gamma_0)^{p^{j-1}}.
    \end{align*}
    By our choice of $\gamma_0$, we have that 
    \[    \lim_{j\to\infty}\fint_{B_j}\chi_E\,d\mu=\lim_{j\to\infty}\frac{\mu(B_j\cap E)}{\mu(B_j)}=0.
    \]
    Since $z\in B(x_0,R_0/4)$ is arbitrary, it follows from the Lebesgue differentiation theorem that $\mu(B(x_0,R_0/4)\cap E)=0$.  However, this is a contradiction, as $x_0\in\partial E$, and we have assumed that $\mu(B(x_0,R)\cap E)>0$ for all $R>0$.  Hence,
    \[
    \frac{\mu(B(x_0,R_0)\cap E)}{\mu(B(x_0,R_0))}\ge\gamma_0.\qedhere
    \]
\end{proof}

\begin{remark}\label{rem:Connectedness}
    One can obtain the conclusion of Theorem~\ref{thm:Uniform Density} by replacing the assumption that $X$ is connected with the weaker assumption that $X$ is uniformly perfect.  A uniformly perfect space equipped with a doubling measure satisfies the reverse doubling condition, see \cite{H} for example, and so Theorem~\ref{thm:FractionalPoincare} and Lemma~\ref{lem:FracMazya} are still applicable under this assumption. Uniform perfectness also yields an estimate similar to \eqref{eq:Connectedness Application}.  Under this assumption, however, the constant $\gamma_0$ will depend additionally upon the uniform perfectness constant.

    Moreover, since we are assuming that $\Omega$ is a domain, and hence connected, we can remove the assumption of uniform perfectness of $X$ entirely, provided we consider $x_0\in\partial E\cap\Omega$ and $R_0>0$ such that $B(x_0,2R_0)\subset\Omega$ and $\Omega\setminus B(x_0,2R_0)\ne\varnothing$.  In this case, \eqref{eq:Connectedness Application} follows by the same argument in the proof of Theorem~\ref{thm:Uniform Density}.   Furthermore, an examination of the proof of Theorem~\ref{thm:FractionalPoincare} in \cite{DLV} shows that if there exists a constant $0<c<1$ so that 
\[
\mu(B(x,r/2))\le c\mu(B(x,r))
\]
for all $0<r\le R$, then the desired fractional Poincar\'e inequality holds on $B(x,R)$.  In our case, we want to apply the fractional Poincar\'e inequality on balls $B(z,2r)\subset B(x_0,R_0/2)$. 
 If $B(x_0,2R_0)\subset\Omega$ and $\Omega\setminus B(x_0,2R_0)\ne\varnothing$, then $B(z,2r)\subset\Omega$ and $\Omega\setminus B(z,2r)\ne\varnothing$.  In this case, connectedness of $\Omega$ and the argument used to obtain \eqref{eq:Connectedness Application}, gives
 \[
 \frac{\mu(B(z,\rho/2))}{\mu(B(z,\rho))}\le1/(1+C_\mu^{-3})
 \]
 for all $0<\rho<2r$.  Hence, we have the fractional Poincar\'e inequality on the desired ball, and can apply Lemma~\ref{lem:FracMazya} as we did in the above proof.  For simplicity, and since we assume connectedness, via the LLC-1 condition, in Theorem~\ref{thm:Porosity} below, we have kept the assumption of connectedness in the statement of Theorem~\ref{thm:Uniform Density}.
\end{remark}

\subsection{Porosity}
By using Theorem~\ref{thm:Uniform Density}, we now prove the porosity result Theorem~\ref{thm:Porosity}.

\begin{proof}[Proof of Theorem~\ref{thm:Porosity}]
Let $E\subset X$ be a minimizer of $\J_\Omega^s$, possibly modified on a set of measure zero so that the conclusion to Theorem~\ref{thm:Uniform Density} holds.  Let $x_0\in\partial E\cap \Omega$ and $R_0>0$ be such that $B(x_0,2R_0)\subset\Omega$.  We prove the first containment of \eqref{eq:Porosity}, as the second follows from the first and the fact that the complement of a minimizer of $\J_\Omega^s$ is a minimizer.  

For each $R_0/2\le r\le R_0$, let 
$$\wtil\rho_r=\inf\{\rho>0: B(y,\rho)\setminus E\ne\varnothing\text{ for all }y\in B(x_0,r)\cap E\}.$$
Let $C_L\ge 1$ denote the constant from the LLC-1 condition, see Definition~\ref{def:LLC}. If $\wtil\rho_r>r/(20C_L)$, then there exists $y\in B(x_0,R_0)\cap E$ such that $B(y,R_0/(40C_L))\subset E\cap\Omega$, and so we can take $C=40C_L$.  Thus, without loss of generality, we assume that $\wtil\rho_r\le r/(20C_L)$ for each $R_0/2\le r\le R_0$.  

Let $0<\eps<R_0/(12C_L)$, and let $\rho_r:=2C_L(\wtil\rho_r+\eps)$.  By the 5-covering lemma, we can cover $B(x_0,r)\cap E$ by balls $\{B(y_i,5\rho_r)\}_i$, with $y_i\in B(x_0,r)\cap E$, such that the collection $\{B(y_i,\rho_r)\}_i$ is pairwise disjoint. By Theorem~\ref{thm:Uniform Density} and doubling, we have that 
\begin{align}\label{**}
\gamma_0\mu(B(x_0,r))\le\mu(B(x_0,r)\cap E)\le\sum_i\mu(B(y_i,5\rho_r))\lesssim\sum_i\mu(B(y_i,\rho_r)).
\end{align}

By the definition of $\wtil\rho_r$, there exists $z\in B(y_i,\wtil\rho_r+\eps)\setminus E$, and by the LLC-1 condition, we can join $y_i$ and $z$ by a connected set inside the $B(y_i,C_L(\wtil\rho_r+\eps))$.  Therefore, there exists $z'\in\partial E\cap B(y_i,C_L(\wtil\rho_r+\eps))=\partial E\cap B(y_i,\rho_r/2)$.  By our choices, we have that $B(z',\rho_r)\subset\Omega$, and so by Theorem~\ref{thm:Uniform Density} and the doubling property of $\mu$, 
\begin{align}\label{eq:Density Application 1}
    \frac{\mu(B(y_i,\rho_r)\cap E)}{\mu(B(y_i,\rho_r))}\gtrsim\frac{\mu(B(z',\rho_r/2)\cap E)}{\mu(B(z',\rho_r/2))}\ge\gamma_0.
\end{align}
Similarly, Theorem~\ref{thm:Uniform Density} gives us
\begin{equation}\label{eq:Density Application 2}
\frac{\mu(B(y_i,\rho_r)\setminus E)}{\mu(B(y_i,\rho_r))}\gtrsim\gamma_0.
\end{equation}
By \eqref{eq:LsKernel}, the doubling property of $\mu$, as well as \eqref{eq:Density Application 2} and \eqref{eq:Density Application 1}, it follows that 
\begin{align*}
L_s(B(y_i,\rho_r)\cap E,\,B(y_i,\rho_r)\setminus E)&\simeq\int_{B(y_i,\rho_r)\cap E}\int_{B(y_i,\rho_r)\setminus E}\frac{1}{d(x,y)^s\mu(B(x,d(x,y)))}d\mu(y)d\mu(x)\\	
	&\gtrsim \frac{1}{\rho_r^s}\int_{B(y_i,\rho_r)\cap E}\frac{\mu(B(y_i,\rho_r)\setminus E)}{\mu(B(x,2\rho_r))}d\mu(x)\\
	&\gtrsim \frac{\gamma_0}{\rho_r^s}\int_{B(y_i,\rho_r)\cap E}\frac{\mu(B(y_i,\rho_r))}{\mu(B(x,2\rho_r))}d\mu(x)\\
	&\gtrsim \gamma_0\frac{\mu(B(y_i,\rho_r)\cap E)}{\rho_r^s}\gtrsim \gamma_0^2\frac{\mu(B(y_i,\rho_r))}{\rho_r^s}.
\end{align*}
From \eqref{**}, it then follows that 
\begin{align}\label{*}
\sum_i L_s(B(y_i,\rho_r)\cap E,B(y_i,\rho_r)\setminus E)&\gtrsim\frac{\gamma_0^2}{\rho_r^s}\sum_i\mu(B(y_i,\rho_r))\gtrsim\gamma_0^3\frac{\mu(B(x_0,r))}{\rho_r^s}.
\end{align}

By Lemma~\ref{lem:SupSubSoln}, it follows that 
\[
L_s(B(x_0,2r)\cap E,X\setminus E)\lesssim L_s(B(x_0,2r)\cap E, E\setminus B(x_0,2r)),
\]
and so by disjointness of $\{B(y_i,\rho_r)\}_i$, we obtain 
\begin{align}\label{***}
\sum_i  L_s(B(y_i,\rho_r)\cap E,B(y_i,\rho_r)&\setminus E)\le\sum_i  L_s(B(y_i,\rho_r)\cap E,X\setminus E)\nonumber\\
&\le L_s(B(x_0,2r)\cap E,X\setminus E)\nonumber\\
	&\lesssim L_s(B(x_0,2r)\cap E, E\setminus B(x_0,2r))\le L_s(B(x_0,2r),X\setminus B(x_0,2r)).
\end{align}
For each $x\in B(x_0,2r),$ let $r_x:=2r-d(x_0,x)$.  By the doubling property of $\mu$, we have that 
\begin{align*}
\int_{X\setminus B(x_0,2r)}\frac{d\mu(y)}{d(x,y)^s\mu(B(x,d(x,y)))}&\le\int_{X\setminus B(x,r_x)}\frac{d\mu(y)}{d(x,y)^s\mu(B(x,d(x,y)))}\\
    &=\sum_{m=0}^\infty\int_{B(x,2^{m+1}r_x)\setminus B(x,2^mr_x)}\frac{d\mu(y)}{d(x,y)^s\mu(B(x,d(x,y)))}\\
    &\lesssim\frac{1}{(2r-d(x_0,x))^s},
\end{align*}
from which we obtain, by \eqref{eq:LsKernel},
\begin{align*}
L_s(B(x_0,2r), X\setminus B(x_0,2r))&\simeq\int_{B(x_0,2r)}\int_{X\setminus B(x_0,2r)}\frac{d\mu(y)d\mu(x)}{d(x,y)^s\mu(B(x,d(x,y)))}\\
	&\lesssim\int_{B(x_0,2r)}\frac{1}{(2r-d(x,x_0))^s}d\mu(x).
\end{align*}
Combining this inequality with \eqref{*} and \eqref{***}, and noting that $\rho_r\le\rho_{R_0}$ by definition, we obtain by the doubling property of $\mu$,
$$\gamma_0^3\frac{\mu(B(x_0,R_0))}{\rho_{R_0}^s}\lesssim \int_{B(x_0,2r)}\frac{1}{(2r-d(x,x_0))^s}d\mu(x)$$
for each $R_0/2\le r\le R_0$. 
 Integrating both sides with respect to $r$, and using Tonelli's theorem and the doubling property of $\mu$, we have that
\begin{align*}
R_0\gamma_0^3\frac{\mu(B(x_0,R_0))}{\rho_{R_0}^s}&\lesssim \int_0^{R_0}\int_{B(x_0,2r)}\frac{1}{(2r-d(x,x_0))^s}d\mu(x)dr\\
    &=\int_0^{R_0}\int_{B(x_0,2R_0)}\frac{\chi_{B(x_0,2r)}(x)}{(2r-d(x,x_0))^s}d\mu(x)dr\\
    &=\int_{B(x_0,2R_0)}\int_{d(x,x_0)/2}^{R_0}\frac{1}{(2r-d(x,x_0))^s}drd\mu(x)\lesssim R_0^{1-s}\mu(B(x_0,R_0)).
\end{align*}
Therefore, there exists a constant $C_0$ depending only on $\gamma_0$, $s$, and the doubling constant of $\mu$ such that $\rho_{R_0}\ge R_0/C_0$, and so  $\wtil\rho_{R_0}\ge R_0/(2C_LC_0)$, since $0<\eps<R_0/(12C_L)$ is arbitrary.  Hence, we obtain the desired constant by setting $C:=4C_LC_0$.
\end{proof}

The conclusion of Theorem~\ref{thm:Porosity} immediately tells us that if $E$ is a minimizer of $\J_\Omega^s$, modified on a set of measure zero if necessary, then $\Ha^{-s}(\partial E\cap\Omega)<\infty$.  Indeed, for $\delta>0$, consider a countable cover $\{B(x_i,r_i)\}_i$ of $\partial E\cap\Omega$ such that $x_i\in\partial E$, $r_i<\delta$ and $B(x_i,2r_i)\subset\Omega$.  The 5-covering lemma then gives us a pairwise disjoint subcollection $\{B(x_i,2r_i)\}_{i\in I\subset\N}$ such that $\partial E\cap\Omega\subset\bigcup_{i\in I}B(x_i,10r_i)$.  For each $i\in I$, there exists $y_i,z_i\in B(x_i,r_i)$ such that $B^1_i:=B(y_i,r_i/C)\subset E\cap\Omega$ and $B^2_i:=B(z_i,r_i/C)\subset\Omega\setminus E$, by Theorem~\ref{thm:Porosity}.  By the doubling property of $\mu$, it then follows that 
\begin{align*}
    \infty>\J_\Omega^s(E)\ge L_s(E\cap\Omega,\Omega\cap E)&\ge\sum_{i\in I}L_s(B(x_i,2r_i)\cap E,B(x_i,2r_i)\setminus E)\\
    &\ge\sum_{i\in I}L_s(B^1_i,B^2_i)\gtrsim\sum_{i\in I}\frac{\mu(B(x_i,r_i))}{r_i^s}\ge\Ha^{-s}_\delta(\partial E\cap\Omega).
\end{align*}
However, the conclusion of Theorem~\ref{thm:Porosity} together with the doubling property of $\mu$ implies that $\partial E\cap\Omega\subset\partial^*E$.  Therefore, Theorem~\ref{thm:CodimHausZero} gives us the following stronger corollary:

\begin{cor}\label{cor:Minimizer Haus}
Let $(X,d,\mu)$ be a compact metric measure space satisfying the LLC-1 condition, with $\mu$ a doubling measure. Let $\Omega\subset X$ be a bounded domain, let $0<s<1$, and let $E$ be a minimizer of $\J_\Omega^s$ such that $\chi_E\in B^s_{1,1}(X)$.  Then, after modifying $E$ on a set of measure zero if necessary, we have that $\Ha^{-s}(\partial E\cap\Omega)=0$. 
\end{cor}

\end{document}